\documentclass[dvipsnames]{amsart}
\usepackage{ytableau,amsthm,amssymb,amsmath,mathtools,tikz}
\usetikzlibrary{arrows.meta}
\usepackage{mathdots}
\usepackage[margin=1.25in]{geometry}
\usepackage{hyperref}

\newtheorem{theorem}{Theorem}[section]
\newtheorem{lemma}[theorem]{Lemma}
\newtheorem{corollary}[theorem]{Corollary}
\newtheorem{proposition}[theorem]{Proposition}

\newtheorem{claim}[theorem]{Claim}

\theoremstyle{definition}
\newtheorem{definition}[theorem]{Definition}
\newtheorem{definition-theorem}[theorem]{Definition-Theorem}
\newtheorem{example}[theorem]{Example}
\newtheorem{remark}[theorem]{Remark}

\numberwithin{equation}{section}

\newcommand{\SVT}{\mathrm{SVT}}
\newcommand{\PSVT}{\mathrm{PSVT}}
\newcommand{\QSVT}{\mathrm{QSVT}}

\newcommand{\Mat}{\mathrm{Mat}}
\newcommand{\SSX}{X^{\mathrm{ss}}}
\newcommand{\SSM}{\operatorname{Mat}^{\text{ss}}}

\DeclareMathOperator{\height}{\mathrm{ht}}

\newcommand{\groth}{\mathfrak{G}}

\newcommand{\grothsp}{\groth^{\mathsf{Sp}}}
\newcommand{\IN}{GP}
\newcommand{\INQ}{GQ}

\newcommand{\GL}{\mathrm{GL}}
\newcommand{\Sp}{\mathrm{Sp}}
\newcommand{\Orth}{\mathrm{O}}
\newcommand{\flags}{\mathcal{F}lag}

\DeclareMathOperator{\rank}{\mathrm{rank}}
\DeclareMathOperator{\pf}{\mathrm{pf}}
\DeclareMathOperator{\reg}{\mathrm{reg}}

\newcommand{\BBB}{\mathsf{B}}
\newcommand{\CCC}{\mathsf{C}}

\newcommand{\symm}{\mathbf{S}}
\newcommand{\FPF}{\mathbf{FPF}}
\newcommand{\code}{\mathtt{BCode}}
\newcommand{\Spcode}{\mathtt{SpCode}}
\newcommand{\shape}{\lambda_{\mathsf{B}}}
\newcommand{\Spshape}{\lambda_{\mathsf{Sp}}}
\newcommand{\kpartial}{\overline{\partial}}

\newcommand{\bad}{\textrm{bad}}

\title[Degrees of $P$-Grothendieck polynomials and regularity of Pfaffian varieties]{Degrees of $P$-Grothendieck polynomials and regularity of Pfaffian varieties}
\author{Oliver Pechenik}
\address[OP]{\parbox{\linewidth}{Dept.\ of Combinatorics \& Optimization, University of Waterloo, Waterloo, ON, N2L 3G1, Canada}}
\email{\parbox[t]{\linewidth}{oliver.pechenik@uwaterloo.ca}}
\author{Matthew St.Denis}
\address[MS]{\parbox{\linewidth}{Dept.\ of Combinatorics \& Optimization, University of Waterloo, Waterloo, ON, N2L 3G1, Canada}}
\email{\parbox[t]{\linewidth}{mstdenis@uwaterloo.ca}}
\date{\today}

\AtBeginDocument{%
   \def\MR#1{}
}

\begin{document}

\begin{abstract}
We prove a formula for the degrees of Ikeda and Naruse's $P$-Grothendieck polynomials using combinatorics of shifted tableaux. We show this formula can be used in conjunction with results of Hamaker, Marberg, and Pawlowski to obtain an upper bound on the Castelnuovo--Mumford regularity of certain Pfaffian varieties known as vexillary skew-symmetric matrix Schubert varieties.
Similar combinatorics additionally yields a new formula for the degree of Grassmannian Grothendieck polynomials and the regularity of Grassmannian matrix Schubert varieties, complementing a 2021 formula of Rajchgot, Ren, Robichaux, St.\ Dizier, and Weigandt.
\end{abstract}

\maketitle{}
\section{Introduction}\label{intro}
There has been much interest recently in the degrees on \emph{Grothendieck polynomials} (e.g., \cite{DMS,Hafner,PSW,Pan.Yu,RRR,RRW,Robichaux}) and related polynomials derived from the combinatorics of $K$-theoretic Schubert calculus. Beyond the intrinsic combinatorial interest of these formulas, a major motivation is that, as first observed in \cite{RRR}, these degrees yield the \emph{Castelnuovo--Mumford regularity} of \emph{matrix Schubert varieties} $X_w$ and certain \emph{Kazhdan--Lusztig varieties} $X_{u,w}$. These affine varieties, introduced by \cite{Fulton} and \cite{Woo.Yong} respectively, are important models for local properties of Schubert varieties. Moreover, $X_w$ and $X_{u,w}$ are generalized determinantal varieties, and in this context various special cases have a long and distinguished history in commutative algebra (see, e.g., \cite{Hochster.Eagon,Narasimhan, Abhyankar,Herzog.Trung, Conca, Conca.Herzog,Gonciulea.Miller, Ghorpade.Krattenthaler, Conca.DeNegri.Gorla}). In Schubert calculus, representatives for the cohomology and $K$-theory classes of Schubert varieties may be obtained as torus-equivariant classes of $X_w$ (e.g., \cite{Fulton,Feher.Rimanyi,Knutson.Miller}), and associated combinatorial formulas may then be obtained from studying Gr\"obner bases (e.g., \cite{Knutson.Miller,Hamaker.Pechenik.Weigandt,Klein.Weigandt}). Castelnuovo--Mumford regularity, meanwhile, is an important algebraic invariant that measures the complexity of the syzygies of defining ideals and simplifies Gr\"obner basis calculations; for a survey of regularity, see \cite{Chardin}.  

Previous work on the degrees of Grothendieck polynomials has been restricted to those in ``type A'', related to the Schubert varieties in \emph{Grassmannians} $\GL_n(\mathbb{C})/P$ and the \emph{complete flag variety} $\GL_n(\mathbb{C})/B$. (Here, $B$ denotes the \emph{Borel subgroup} of upper triangular matrices and $P \supset B$ is a \emph{maximal parabolic subgroup}.) 
In contrast, we study the degrees of \emph{$P$-Grothendieck polynomials} $\IN_\lambda$, which are related in one sense to ``type B'' Schubert calculus and in another to ``type C''. (We note that the ``$P$'' in the name of $P$-Grothendieck polynomials does not refer to a parabolic subgroup, but is instead just the letter ``$P$''.) 

Our main combinatorial theorem is the following.
\begin{theorem} \label{main}
Let $\lambda$ be a strict partition and let $\IN_{\lambda,n}$ be the $P$-Grothendieck polynomial for $\lambda$ in $n$ variables. Let $\Delta = (\Delta_1,\dots,\Delta_\ell)$ be the largest partition contained in $\lambda$ such that all parts of $\Delta$ differ by at least two. Then the degree of the $P$-Grothendieck polynomial is
\[
\deg(\IN_{\lambda,n}) = \begin{cases}
 	|\Delta| + 2n\ell-\ell^2-\ell, & \text{if } \Delta_\ell > 1; \\
 	|\Delta|+2n\ell-\ell^2-n, & \text{if }  \Delta_\ell = 1.
 \end{cases}
\]
\end{theorem}
Our proof of Theorem~\ref{main} is based on direct combinatorial analysis of the tableau formula for $\IN_{\lambda,n}$ given in \cite{IN}. The first step of the proof is a reduction from general $\IN_{\lambda,n}$ to a special subclass $\IN_{\Delta,n}$, while the second step calculates the degree of $\IN_{\Delta,n}$ by explicitly identifying a tableau contributing to its highest-degree component.

$P$-Grothendieck polynomials were introduced by T.~Ikeda and H.~Naruse \cite{INN,IN} as representatives for $K$-theoretic Schubert classes on \emph{maximal orthogonal Grassmannians} $\Orth_{2n+1}(\mathbb{C})/P$, where $P$ is a particular parabolic subgroup of $\Orth_{2n+1}(\mathbb{C})$. It is in this sense that $\IN_\lambda$ is ``type B''. By specializing $\IN_\lambda$, one obtains the classical $P$-Schur functions $P_\lambda$ that were introduced by I.~Schur \cite{Schur} to describe the projective representation theory of symmetric groups and were connected to the cohomological Schubert calculus of maximal orthogonal Grassmannians by P.~Pragacz \cite{Pragacz}.

Recent work of E.~Marberg and B.~Pawlowski \cite{Marberg.Pawlowski:properties} shows that $P$-Grothendieck polynomials coincide with the stable limits of \emph{vexillary symplectic Grothendieck polynomials}, representatives for $K$-theoretic classes of $\Sp_n(\mathbb{C})$-orbit-closures on the flag variety $\GL_n(\mathbb{C})/B$. It is in this sense that $\IN_\lambda$ is related to ``type C''. Just as matrix Schubert varieties provide affine models of Schubert varieties, \emph{skew-symmetric matrix Schubert varieties} $\SSX_w$ provide affine models for $\Sp_n(\mathbb{C})$-orbit-closures (see \cite{Wyser.Yong}). Marberg and Pawlowski \cite{Marberg.Pawlowski:formulas} have shown that, just as ordinary Grothendieck polynomials arise as torus-equivariant classes of ordinary matrix Schubert varieties, symplectic Grothendieck polynomials arise as torus-equivariant classes of skew-symmetric matrix Schubert varieties $\SSX_w$. Skew-symmetric matrix Schubert varieties are instances of Pfaffian varieties \cite{Marberg.Pawlowski:Grobner} and hence have a commutative algebra history in that context (see, e.g., \cite{Herzog.Trung,DeNegri.Gorla:liaison,DeNegri.Gorla:invariants,DeNegri.Sbarra,Lorincz.Raicu.Walther.Weyman}). Indeed, we suspect that vexillary skew-symmetric matrix Schubert varieties are related to the \emph{mixed ladder Pfaffian varieties} of \cite{DeNegri.Gorla:liaison}, for which \cite{DeNegri.Gorla:invariants} determine a recursive regularity formula; forthcoming work by L.~Escobar, A.~Fink, J.~Rajchgot, and A.~Woo \cite{Escobar.Fink.Rajchgot.Woo} will clarify this relationship.

To obtain algebraic consequences of Theorem~\ref{main}, we leverage the connection between $\IN_\lambda$ and \emph{symplectic Grothendieck polynomials}.  We observe that $\SSX_w$ is Cohen--Macaulay and that the degrees of symplectic Grothendieck polynomials yield its regularity, so that, after describing the relation between degrees of symplectic Grothendieck polynomials and of $P$-Grothendieck polynomials, Theorem~\ref{main} implies the following, which is our main algebraic result.
\begin{theorem} \label{regularity_theorem}
Let $z \in \symm_n$ be an FPF-vexillary fixed-point-free involution with associated symplectic shape the partition $\Spshape(z) = \lambda$. Let $k$ be the position of the last nonzero entry of the symplectic code $\Spcode(z)$ and let $\Delta$ be the largest partition contained in $\lambda$ such that all parts differ by at least two. 

Then the Castelnuovo--Mumford regularity of the vexillary skew-symmetric matrix Schubert variety 
$\SSX_z$ satisfies
\[\reg(\SSX_z) \leq \begin{cases}
 	2k\ell-\ell^2-\ell - \big(|\Spshape(z)| - |\Delta|\big), & \text{if } \Delta_\ell > 1; \\
 	2k\ell-\ell^2-k  - \big(|\Spshape(z)| - |\Delta|\big), & \text{if }  \Delta_\ell = 1.
 \end{cases}\]
\end{theorem}

In contrast to earlier ``type A'' work (e.g., \cite{RRR,PSW}), we obtain only an upper bound on regularity instead of an exact formula. This feature arises from the fact that we currently do not have a formula for the degrees of symplectic Grothendieck polynomials, but rather only partial information extracted from such a formula for $\IN_\lambda$. It would be very interesting to study the degrees of symplectic Grothendieck polynomials directly. And indeed one might hope to be able to do so, perhaps using the combinatorics developed in \cite{Marberg.Pawlowski:formulas,Hamaker.Marberg.Pawlowski:pipe,Marberg.Pawlowski:Grobner}.

In Section~\ref{symdegrees}, we imitate the proof of Theorem~\ref{main} to prove a new formula (Theorem~\ref{thm:typeAdegree}) for the degrees of \emph{symmetric Grothendieck polynomials}. These polynomials represent $K$-theoretic Schubert classes on a Grassmannian $\GL_n(\mathbb{C})/P$. As shown in \cite{RRR}, the degree of a symmetric Grothendieck polynomial yields the Castelnuovo--Mumford regularity of a corresponding Grassmannian matrix Schubert variety; hence, our degree formula yields a new regularity formula (Corollary~\ref{cor:typeAreg}) in this context. Our formula differs than those given previously by \cite{RRR,Hafner,RRW,PSW} and it is unclear how to relate it to those earlier formulas. (However, Proposition~\ref{prop:anna_sez} and surrounding discussion give some tentative connections to the formulas of \cite{PSW}.) Moreover, our proof is arguably somewhat easier. Except for some background on matrix Schubert varieties in Section~\ref{sec:MSV}, Section~\ref{symdegrees} can be read independently of Sections~\ref{degrees} and~\ref{regularity} by readers whose primary interest is this type A setting.

Finally, in Section~\ref{Q}, we discuss obstacles to proving an analogue of Theorem~\ref{main} for \emph{$Q$-Grothendieck polynomials} $\INQ_\lambda$. It is perhaps surprising that this analogous problem should be substantially more difficult, since the classical \emph{$Q$-Schur polynomials} obtained by specializing $\INQ_\lambda$ barely differ from $P$-Schur polynomials. Nonetheless, there are significant technical difficulties and we provide only partial progress. A relevant fact may be that $Q$-Grothendieck polynomials are related to the $K$-theoretic Schubert calculus of \emph{Lagrangian Grassmannians} $\Sp_n(\mathbb{C})/P$ (see \cite{IN}) and likely related to the \emph{symmetric matrix Schubert varieties} that give affine models for $\Orth_n(\mathbb{C})$-orbit-closures on $\GL_n(\mathbb{C})/B$; the geometry of all these varieties is known to be significantly harder and worse-behaved than those of the varieties associated to $\IN_\lambda$ (see, e.g., \cite{Pin,Buch.Ravikumar,Pechenik.Yong,Wyser.Yong,Marberg.Pawlowski:Grobner} for related remarks).

The earlier sections of this paper are organized as follows. Section~\ref{sec:background} recalls combinatorial and algebraic background. Section~\ref{degrees} establishes Theorem~\ref{main}. Section~\ref{regularity} recalls additional background related to (skew-symmetric) matrix Schubert varieties and establishes Theorem~\ref{regularity_theorem}.

\section{Background}\label{sec:background}

\subsection{Combinatorial background: Tableaux and Grothendieck polynomials}
Diagrams of partitions are drawn in French notation, with the largest part at the bottom of the diagram. 
A partition $\lambda = (\lambda_1, \dots ,\lambda_n)$ is called \textit{strict} if $\lambda_{i-1} -\lambda_i \geq 1$ for all $i$. 
If $\lambda$ is a strict partition, then the \textit{shifted diagram} of $\lambda$ is obtained by appending $i-1$ empty spaces to the left of the $i$th row of the diagram of $\lambda$. The \textit{main diagonal} of a shifted diagram consists of the leftmost box in each row. Below is the diagram (left) and shifted diagram (right) of the partition $6421$, with the main diagonal shaded in red \textcolor{red}{$\blacksquare$}:
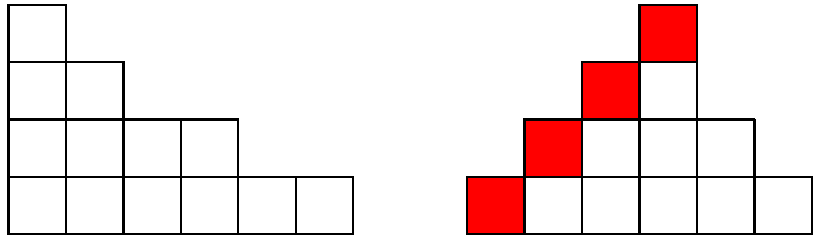
\[
\ytableausetup{boxsize=.75cm}
\begin{ytableau}
\\
&\\
& & &\\
& & & & &
\end{ytableau}
\hspace{1.5cm}
\begin{ytableau}
\none & \none &\none &*(red)\\
\none &\none & *(red)&\\
\none & *(red)& & &\\
*(red)& & & & &
\end{ytableau}
\]

If $\lambda_{i-1}-\lambda_i \geq 2$ for all $i$, then $\lambda$ is called a \emph{D-partition} (from the French: \textit{Diff\'erence}-partition; see~\cite{Viennot}). We draw the diagram of a D-partition without further shifting of the rows.

Given a partition $\lambda$, we write $\mathsf{B} \in \lambda$ to denote a specific box of the diagram. To ease manipulations of adjacent boxes, we use the notations $\mathsf{B}^\uparrow,\mathsf{B}^\rightarrow,\mathsf{B}^\downarrow,\mathsf{B}^\leftarrow$ to denote the boxes immediately above, to the right, below, and to the left of $\mathsf{B}$, respectively. When there is no box above $\mathsf{B}$ in $\lambda$, it is convenient to simply allow $\mathsf{B}^\uparrow$ to be an empty box whose contents are the empty set, to avoid the repetitive need to qualify statements with ``provided such a box exists''. We make identical definitions in the other directions.

A \textit{filling} of a diagram $\lambda$ with a totally ordered alphabet $\mathbb{A}$ is a function $T: \lambda \rightarrow \mathbb{A}$, which ``fills" each box $\mathsf{B} \in \lambda$ with an element of $\mathbb{A}$. A \textit{set-valued} filling $T$ assigns to each $\mathsf{B} \in \lambda$ a nonempty subset $T(\mathsf{B}) \in \mathcal{P}(\mathbb{A}) \setminus \varnothing$. Ordinary diagrams are filled by the natural numbers $\mathbb{Z}_{>0}$, while shifted diagrams are instead filled using the alphabet
\[\mathbb{S} =  1' < 1 < 2' < 2 < 3' < 3 < \cdots  .\]
We use interval notation to refer to subsets of the alphabets $\mathbb{Z}_{>0}$ and $\mathbb{S}$:
\[[a,b] = \{c \in \mathbb{Z}_{>0}: a \leq c \leq b\}, \quad [a,b]_{\mathbb{S}} = \{c \in \mathbb{S}: a \leq c \leq b\}.\]
Following \cite{Buch}, we call a set-valued filling $T$ of an ordinary diagram $\lambda$ a \emph{(semistandard) set-valued tableau} if $T$ satisfies the following two properties for all $\mathsf{B} \in \lambda$:

\begin{itemize}
\item{$\max(T(\mathsf{B})) < \min(T(\mathsf{B}^{\uparrow}))$},
\item{$\max(T(\mathsf{B})) \leq \min(T(\mathsf{B}^{\rightarrow}))$.}
\end{itemize}

The \emph{content vector} $c(T)$ of a set-valued tableau is defined as
\[c(T) = (\textup{\# of ones in } T,\text{\# of twos in }T,\dots)\]
and we say the \emph{degree} of $T$ is $d(T) = \sum c(T)$.

\begin{example}\label{ex:SVT} The following is a semistandard set-valued tableau of shape $\lambda = (6,4,2,1)$, with content $(1,3,4,3,5,3)$ and degree $19$.
\[
\begin{ytableau}
56\\
4 & 6\\
3 & 345 & 5 & 5\\
12 & 2 & 23 & 3 & 45 & 6
\end{ytableau}
\]
\end{example}

We denote the set of all semistandard set-valued tableaux on $\lambda$ by $\SVT(\lambda)$, and the subset with entries from $[1,n]$ by $\SVT(\lambda,n)$.

\begin{definition}[{\cite{Lascoux.Schutzenberger:Hopf,Buch}}] \label{grothendieckdef}
The \emph{symmetric Grothendieck polynomial} for $\lambda$ in $n$ variables is the polynomial
\begin{equation}
G_{\lambda,n} = \sum_{T \in \SVT(\lambda,n)}(-1)^{d(T)-|\lambda|}\mathbf{x}^{c(T)}.
\end{equation}
\end{definition}
For example, the tableau of Example~\ref{ex:SVT} contributes the term $(-1)^{19-13} \cdot x_1 x_2^3 x_3^4 x_4^3 x_5^5 x_6^3$ to the polynomial $G_{6421,k}$ for any $k \geq 6$. The polynomial $G_{\lambda,n}$ is symmetric in the variables $x_1, \dots, x_n$; for a bijective proof of this fact, see \cite{Ikeda.Shimazaki}.

\bigskip
The analogous definitions we will need for shifted diagrams are as follows.

\begin{definition} \label{ptabdef}
A set-valued filling $T$ of $\lambda$ is a \emph{$P$-shifted set-valued tableau} if for all boxes $\mathsf{B}$ in the shifted diagram of $\lambda$:
\begin{itemize}
\item{If $\max(T(\mathsf{B}))$ is primed, then 
\begin{itemize}
\item{$\max(T(\mathsf{B})) \leq \min(T(\mathsf{B}^\uparrow))$.}
\item{$\max(T(\mathsf{B})) < \min(T(\mathsf{B}^\rightarrow))$.}
\end{itemize}
}
\item{If $\max(T(\mathsf{B}))$ is unprimed, then 
\begin{itemize}
\item{$\max(T(\mathsf{B})) < \min(T(\mathsf{B}^\uparrow))$.}
\item{$\max(T(\mathsf{B})) \leq \min(T(\mathsf{B}^\rightarrow))$.}
\end{itemize}}
\item{No boxes on the main diagonal of $\lambda$ contain any primed entries.}
\end{itemize}
We denote the set of all $P$-shifted set-valued tableaux on $\lambda$ by $\PSVT(\lambda)$, and the subset with entries from $[1',n]_\mathbb{S}$ by $\PSVT(\lambda,n)$. (In fact, the conditions of Definition~\ref{ptabdef} prevent the symbol $1'$ from appearing.)
\end{definition}
The \emph{content vector} $c(T)$ of a $P$-shifted set-valued tableau $T$ is defined as
\[c(T) = (\textup{\# of ones in } T,\text{\# of twos in }T,\dots)\]
and we say the \emph{degree} of $T$ is $d(T) = \sum c(T)$. We emphasize that neither content nor degree distinguish between primed and unprimed entries.
\begin{example} \label{Pex} The following is a $P$-shifted set-valued tableau of shape $\lambda = 6421$, with content $(2,3,4,6,2)$ and degree $17$.
\[\begin{ytableau}
\none &\none & \none & 5\\
\none & \none & 34 & 45'\\
\none & 2 & 2 & 3' & 4'4\\
1& 1 & 2' & 3' & 34' & 4 
\end{ytableau}\]
\end{example}

\begin{definition}[\cite{INN,IN}] \label{pgrothendieckdef}
The \emph{$P$-Grothendieck polynomial} for $\lambda$ in $n$ variables is the polynomial
\begin{equation}
\IN_{\lambda,n} = \sum_{T \in \PSVT(\lambda,n)}\beta^{d(T)-|\lambda|}\mathbf{x}^{c(T)}.
\end{equation}
\end{definition}
For example, the tableau of Example~\ref{Pex} contributes the term $(-1)^{17-13} \cdot x_1^2 x_2^3 x_3^4 x_4^6 x_5^2$ to the polynomial $\IN_{6421,k}$ for any $k \geq 5$. For a proof of the symmetry of $P$-Grothendieck polynomials, see \cite{IN}.

\begin{remark}
	The original definition of the $P$-Grothendieck polynomials from \cite{INN,IN}  considers the limit of these polynomials by summing over the full set $\PSVT(\lambda)$. By taking lowest-degree terms, one recovers the classical $P$-Schur functions \cite{Schur}.
\end{remark}

\subsection{Algebraic background: \texorpdfstring{$\mathcal{K}$}{K}-polynomials and Castelnuovo--Mumford regularity} \label{algebraic.background}
Most of the background summarized in this subsection consists of standard commutative algebra, and proofs of all stated facts without their own citations can be found in, for instance, \cite{Eisenbud} or \cite{Miller.Sturmfels}. 

Take $S = \mathbb{C}[x_1, \dots, x_n]$ a polynomial ring with the standard grading $\deg(x_{i}) = 1$, and take $I \subseteq S$ a homogeneous ideal.
For any finitely generated graded $S$-module $M$, we denote the $\mathbb{C}$-vector space of all homogeneous degree $a$ elements of $M$ by $M_a$. Since $M$ is finitely generated, the dimension of $M_a$ is finite for all $a$, and so we can define the \emph{Hilbert series} of $M$ as the formal power series
\[H(M;t) \coloneqq \sum_{a \in \mathbb{Z}_{\geq 0}}\dim_\mathbb{C}(M_a)t^a.\]  
It is often useful to write the Hilbert series as a ratio of two polynomials
\[H(M;t) = \frac{\mathcal{K}(M;t)}{(1-t)^{n}},\]
in which case the polynomial $\mathcal{K}(M;t)$ is referred to as the \emph{$\mathcal{K}$-polynomial} of the module $M$; see \cite{Miller.Sturmfels} for further discussion.

We define the degree-shifted module $M(-j)$ via the condition $M(-j)_a = M_{a-j}$ for all $a$. With this notation, a \emph{free resolution} of $M$ is an exact sequence of graded $S$-modules
\[\dots \rightarrow \bigoplus_{i \in \mathbb{Z}}S(-i)^{b_i^k} \rightarrow \dots \rightarrow \bigoplus_{i \in \mathbb{Z}}S(-i)^{b_i^1} \rightarrow \bigoplus_{i \in \mathbb{Z}}S(-i)^{b_i^0} \rightarrow M \rightarrow 0.\]
A free resolution is called \emph{minimal} if the value of $b_i^j$ is minimized simultaneously for all indices $i,j$. In our situation, there is a unique finite minimal free resolution for any finitely generated graded $M$ by Hilbert's Syzygy Theorem.

Since minimal free resolutions are unique, we can define the \emph{(Castelnuovo--Mumford) regularity} of $M$ as
\[\textup{reg}(M) \coloneqq \max\{i-j:b_i^j \not= 0\}.\] For additional background on regularity, we refer the reader to the survey \cite{Chardin}. 
 The definition of regularity directly offers a kind of bound on the complexity of the free resolution of $M$. We will only be concerned in this paper with the regularity of ideals $I$ and their quotient modules $S/I$, which are essentially the same information, since $\reg I = 1 + \reg (S/I)$. 
 
 In the case of a polynomial ideal $I$, the regularity gives information on the complexity of \emph{Gr\"obner bases} of $I$, via the following theorem of D.~Bayer and M.~Stillman. (For background on Gr\"obner bases and related undefined notions, see \cite{Eisenbud,Cox.Little.OShea,Ene.Herzog}.)
\begin{theorem}[{\cite[Corollary~2.5 \& Proposition~2.11]{Bayer.Stillman}}]\label{thm:BayerStillman}
    Fix a grevlex term order. If $I \subset S$ is a homogeneous ideal in generic coordinates with Castelnuovo--Mumford regularity $m$, then the highest-degree element of a minimal Gr\"obner basis for $I$ has degree exactly $m$.
\end{theorem}
There is a beautiful combinatorial Gr\"obner theory for matrix Schubert varieties, which is quite well developed (e.g., \cite{Knutson.Miller,Knutson.Miller.Yong,Hamaker.Pechenik.Weigandt,Conca.DeNegri.Gorla:Vietnam,Klein,Klein.Weigandt}) and explains algebraically many of the combinatorial formulas for Schubert, Schur, and Grothendieck polynomials. Contrastingly, the analogous theory for skew-symmetric matrix Schubert varieties is understood for only a single special term order \cite{Marberg.Pawlowski:Grobner}. Through Theorem~\ref{thm:BayerStillman}, our Theorem~\ref{regularity_theorem} may provide a hint towards a broader theory.

Computing the regularity of an arbitrary module often requires technical work with free resolutions or local cohomology. However, as noted in \cite{RRR}, when $S/I$ is a Cohen--Macaulay ring, then the following lemma is known (see \cite[Lemma~2.5]{Benedetti.Varbaro} for explanation) and allows one to compute the regularity from the ideal by combinatorial methods.

\begin{lemma}\label{cohenmacreg} Let $I \subseteq S$ and let $S/I$ be a Cohen--Macaulay ring. Then 
\[\textup{reg}(S/I) = \deg (\mathcal{K}(S/I;t))-\height(I).\]
\end{lemma}
For ideals pertaining to Schubert calculus, both terms on the right side of Lemma~\ref{cohenmacreg} can be more approachable than the regularity itself. The height of $I$ can often be computed from the indexing combinatorics for $I$, and the polynomial $\mathcal{K}(S/I;t)$ often has some known combinatorial description as a generating polynomial, such as a Grothendieck polynomial. Thus, the problem of computing the regularity of $I$ can be solved if one is able to compute the degree of the $\mathcal{K}$-polynomial through combinatorial means. The approach provided by this lemma is the underpinning of all combinatorial computations of regularity overviewed in Section~\ref{intro}, and will be followed in this paper as well.

\section{Degrees of \texorpdfstring{$P$}{P}-Grothendieck polynomials and proof of Theorem~\ref{main}}\label{degrees}
By Definition~\ref{pgrothendieckdef}, the degree of the $P$-Grothendieck polynomial $\IN_{\lambda,n}$ is the maximum degree amongst all tableaux $T \in \PSVT(\lambda,n)$. We compute this number in two steps. First, we establish that for all partitions $\lambda$, the degree of $\IN_{\lambda,n}$ is equal to the degree of $\IN_{\Delta,n}$, where $\Delta$ is the largest D-partition  contained in $\lambda$. This reduces the problem to computing the maximum degree of a tableau in $\PSVT(\Delta,n)$, in which case we explicitly construct a tableau of maximum degree, with a sufficiently simple form that the degree can be directly calculated by elementary counting.

\begin{lemma}\label{lem:order_containment}
	Suppose $\mu \subseteq \lambda$ are strict partitions and let $n \geq \ell(\lambda)$. If $T \in \PSVT(\mu,n)$ is a $P$-shifted set-valued tableau, then there exists a $T' \in \PSVT(\lambda,n)$ with $d(T') \geq d(T)$.
\end{lemma}
\begin{proof}
	By induction, it suffices to assume that $|\lambda| - |\mu| = 1$. Let the unique box of $\lambda \setminus \mu$ be $\BBB_0$. We will construct $T'$ from $T$ by filling the box $\BBB_0$ with the value $n$ and then possibly making adjustments to other boxes to remain a valid tableau. For an illustration of the algorithm, see Example~\ref{ex:order_containment}.

 Extending the definition from \cite{Thomas.Yong}, we say a \emph{short ribbon} is an edge-connected set of boxes that does not contain a $2 \times 2$ subshape and where each row and column contains at most two boxes.
Recursively define the boxes of a short ribbon $R$ as follows. For $k$ odd, let $\BBB_k = \BBB_{k-1}^\downarrow$, while for $k$ even let $\BBB_k =	\BBB_{k-1}^\leftarrow$. The short ribbon $R$ consists of all of these boxes $\BBB_k$ for $k \geq 0$.
	We consider two cases according to whether or not $\BBB_0$ lies on the main diagonal.
	
	\bigskip
\noindent
{\sf (Case 1: $\BBB_0$ is not on the main diagonal):}
Add the value $n$ to the box $\BBB_0$. We truncate $R$ to another short ribbon $S$ as follows. First, say $\BBB_1 \in S$ if $T(\BBB_1) \ni n$; otherwise $S$ is the empty ribbon. For each even $j \geq 2$, if $\BBB_{j-1} \in S$ and $|T(\BBB_{j-1})| = 1$ and $T(\BBB_j)  \ni (n-\frac{j}{2}+1)'$, include $\BBB_j \in S$.  For each odd $j \geq 3$, if $\BBB_{j-1} \in S$ and $|T(\BBB_{j-1})| = 1$ and $T(\BBB_j)  \ni n-\frac{(j-1)}{2}$, include $\BBB_j \in S$. 

For each box $\BBB_j \in S$, modify its filling as follows:
\begin{equation}\label{eq:box_shuffle}
    T'(\BBB_j) = \begin{cases}
	T(\BBB_j) \setminus \max T(\BBB_j), & \text{if } |T(\BBB_j)| > 1;\\
	k-1, & \text{if } T(\BBB_j) = k';\\
	k', & \text{if } T(\BBB_j) = k.
\end{cases}
\end{equation}
This operation is well-defined because $n \geq \ell(\lambda)$ by assumption, so the height of the ribbon $S$ is at most $n$, and $k$ cannot become non-positive before the ribbon terminates. Since the ribbon $S$ terminates as soon as it reaches a box with more than one entry, the first case in the definition \eqref{eq:box_shuffle} can only occur at most once; therefore $T'$ deletes at most one entry from $T$, and since we have added one $n$ in the newly added box, it follows that $d(T') \geq d(T)$, as desired. It remains to check that increasingness conditions are satisfied.

Suppose there is an increasingness violation in $T'$; since only the content of the ribbon $S$ was altered, the violation must involve some box of $S$. First, observe that the increasingness violation cannot be between two boxes of $S$ and that it cannot involve a box of $S$ with more than one entry. Let $j$ be the smallest index such that $\BBB_j$ is involved in an increasingness violation. If $j$ is odd, then the violation must be either between $\BBB_j$ and $\BBB_j^\downarrow$ or $\BBB_j$ and $\BBB_j^\rightarrow$. The latter is clearly not possible, since the content of $\BBB_j$ has been decreased from $T$ and these two boxes were increasing in $T$. A violation between $\BBB_j$ and $\BBB_j^\downarrow$ is not possible either, since if $j$ is odd then $T(\BBB_j)$ is, by assumption, unprimed, and so $\max(T(\BBB_j^\downarrow)) < T(\BBB_j)$, which implies $\max(T(\BBB_j^\downarrow)) \leq T'(\BBB_j)$.

Suppose instead then that $j$ is even. Then the only two violations could be between $\BBB_j$ and $\BBB_j^\leftarrow$ or $\BBB_j$ and $\BBB_j^\uparrow$. The latter is impossible, because we have lowered the content of $\BBB_j$ while holding $\BBB_j^\uparrow$ fixed. A violation between $\BBB_j$ and $\BBB_j^\leftarrow$ is also impossible, because $T(\BBB_j)$ is primed by assumption, so $T(\BBB_j^\leftarrow) < T(\BBB_j)$, which implies $T(\BBB_j^\leftarrow) \leq T'(\BBB_j)$.

\bigskip
\noindent
{\sf (Case 2: $\BBB_0$ is on the main diagonal):} 
Add the value $n$ to the box $\BBB_0$. Again, we truncate $R$ to another short ribbon $S$, but in a slightly different way. First, say $\BBB_1 \in S$ if $T(\BBB_1) \ni n$; otherwise $S$ is the empty ribbon. For each even $j \geq 2$, if $\BBB_{j-1} \in S$ and $|T(\BBB_{j-1})| = 1$ and $T(\BBB_j)  \ni n-\frac{j}{2}+1$, include $\BBB_j \in S$.  For each odd $j \geq 3$, if $\BBB_{j-1} \in S$ and $|T(\BBB_{j-1})| = 1$ and $T(\BBB_j)  \ni n-\frac{(j-1)}{2}$, include $\BBB_j \in S$. 

For each box $\BBB_j \in S$, modify its filling as follows:
\begin{equation}\label{eq:box_shuffle2}
   T'(\BBB_j) = \begin{cases}
	T(\BBB_j) \setminus \max T(\BBB_j), & \text{if } |T(\BBB_j)| > 1;\\
	T(\BBB_j) - 1, & \text{if } j \text{ is even};\\
	T(\BBB_j)', & \text{if $j$ is odd}.
\end{cases} 
\end{equation}
Exactly as above, this operation is well-defined and the tableau $T'$ has degree $d(T') \geq d(T)$. We now consider increasingness conditions.

Again as above, any possible increasingness violation must involve exactly one box of the ribbon $S$. Suppose $j$ is the smallest  index on such a box $\BBB_j$. If $j$ is odd, then the only violations can be between $\BBB_j$ and $\BBB_j^\downarrow$ or $\BBB_j$ and $\BBB_j^\rightarrow$. Again, the latter is impossible because we have only decreased the content of $\BBB_j$, and a violation between $\BBB_j$ and $\BBB_j^\downarrow$ is not possible either, since if $j$ is odd then $T(\BBB_j)$ is, by assumption, unprimed, and so $\max(T(\BBB_j^\downarrow)) < T(\BBB_j)$, which implies $\max(T(\BBB_j^\downarrow)) \leq T'(\BBB_j)$.

If on the other hand $j$ is even, then there cannot possibly be any increasingness violations, since $\BBB_j$ is on the main diagonal for even $j$. Any violation would have to involve either $\BBB_j^\leftarrow$ and $\BBB_j^\uparrow$, but neither of these boxes exist for a box on the main diagonal.

Therefore, in either case the tableau constructed by modifying the ribbon $S$ as above and leaving all other boxes unaltered is a $P$-shifted set-valued tableau $T'$ on $\lambda$ with degree greater than or equal to that of $T$, as desired.
\end{proof}

\begin{example}\label{ex:order_containment}
Let $n = 6$, and consider the $P$-shifted set-valued tableau 
\[
\begin{ytableau} 
\none & \none & \none & 5 & 56\\
\none & \none & 4 & 4 & 5' & 6\\
\none & 2 & 3' & 4' & 4 & 5' \\
1 & 12' &  3' & 3 & 3 & 4'4
\end{ytableau}
\]
of shape $\mu = (6,5,4,2)$.

Suppose we wish to construct a $P$-shifted set-valued tableau with greater or equal weight on $\lambda = (6,5,4,3) \supset \mu$. This is performed by the following sequence of steps:

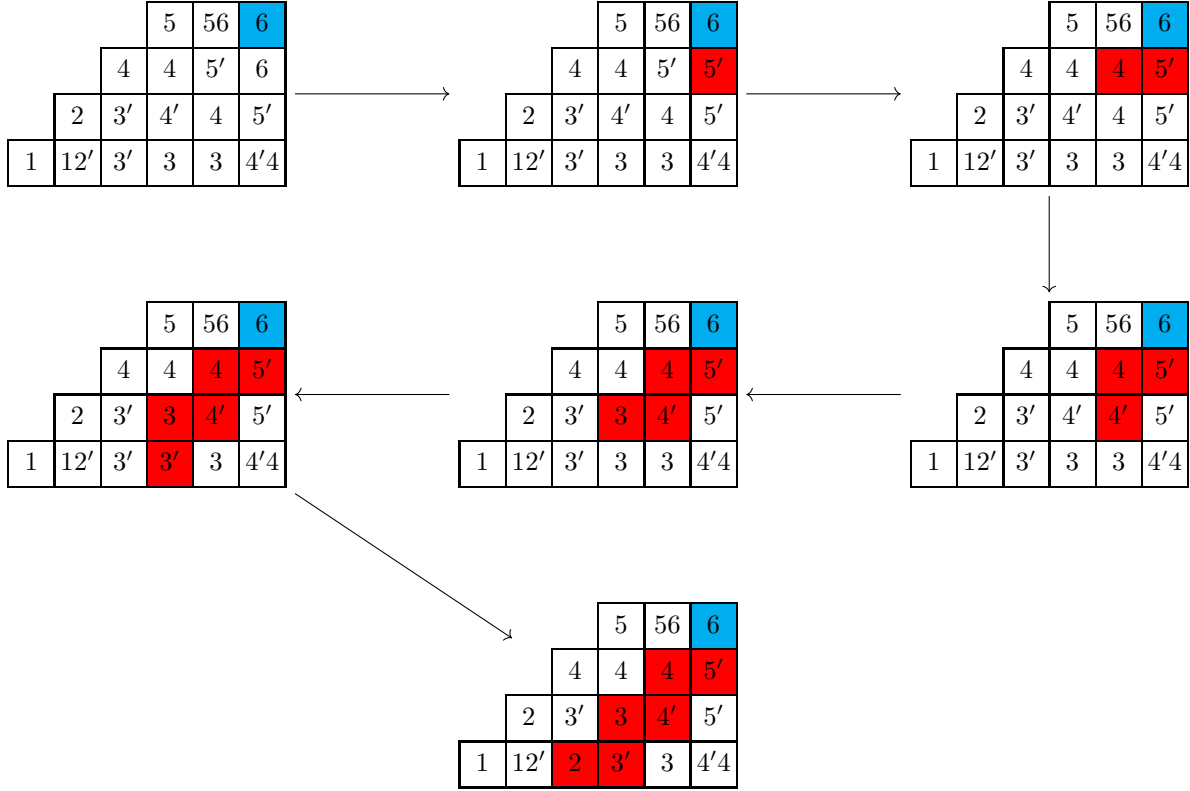
\[
\begin{tikzpicture}
\ytableausetup{boxsize=.6cm}

\node (a) at (0,0) {
\begin{ytableau} 
\none & \none & \none & 5 & 56 & *(Cyan) 6\\
\none & \none & 4 & 4 & 5' & 6\\
\none & 2 & 3' & 4' & 4 & 5' \\
1 & 12' &  3' & 3 & 3 & 4'4
\end{ytableau}
};

\node (b) at (6,0){
\begin{ytableau} 
\none & \none & \none & 5 & 56 & *(Cyan) 6\\
\none & \none & 4 & 4 & 5' & *(red) 5'\\
\none & 2 & 3' & 4' & 4 & 5' \\
1 & 12' &  3' & 3 & 3 & 4'4
\end{ytableau}
};

\node (c) at (12,0) {
\begin{ytableau} 
\none & \none & \none & 5 & 56 & *(Cyan) 6\\
\none & \none & 4 & 4 & *(red) 4 & *(red) 5'\\
\none & 2 & 3' & 4' & 4 & 5' \\
1 & 12' &  3' & 3 & 3 & 4'4
\end{ytableau}
};
\node (d) at (12,-4) {
\begin{ytableau} 
\none & \none & \none & 5 & 56 & *(Cyan) 6\\
\none & \none & 4 & 4 & *(red) 4 & *(red) 5'\\
\none & 2 & 3' & 4' & *(red) 4' & 5' \\
1 & 12' &  3' & 3 & 3 & 4'4
\end{ytableau}
};
\node (e) at (6,-4) {
\begin{ytableau} 
\none & \none & \none & 5 & 56 & *(Cyan) 6\\
\none & \none & 4 & 4 & *(red) 4 & *(red) 5'\\
\none & 2 & 3' & *(red) 3 & *(red) 4' & 5' \\
1 & 12' &  3' & 3 & 3 & 4'4
\end{ytableau}
};
\node (f) at (0,-4) {
\begin{ytableau} 
\none & \none & \none & 5 & 56 & *(Cyan) 6\\
\none & \none & 4 & 4 & *(red) 4 & *(red) 5'\\
\none & 2 & 3' & *(red) 3 & *(red) 4' & 5' \\
1 & 12' &  3' & *(red) 3' & 3 & 4'4
\end{ytableau}
};
\node (g) at (6,-8) {
\begin{ytableau} 
\none & \none & \none & 5 & 56 & *(Cyan) 6\\
\none & \none & 4 & 4 & *(red) 4 & *(red) 5'\\
\none & 2 & 3' & *(red) 3 & *(red) 4' & 5' \\
1 & 12' & *(red) 2 & *(red) 3' & 3 & 4'4
\end{ytableau}
};

\draw[->] (a) -- node[midway,right] {} (b);
\draw[->] (b) -- node[midway,right] {} (c);
\draw[->] (c) -- node[midway,right] {} (d);
\draw[->] (d) -- node[midway,right] {} (e);
\draw[->] (e) -- node[midway,right] {} (f);
\draw[shorten > = -1cm,->] (f) -- node[midway,above left] {} (g);
\end{tikzpicture}.
\]
Here, the unique box $\BBB_0 \in \lambda \setminus \mu$ is shaded in blue \textcolor{Cyan}{$\blacksquare$}, while the boxes of the short ribbon $S$ are shaded in red \textcolor{red}{$\blacksquare$}. In this case, the short ribbon $R$ consists of all of the shaded boxes, either blue or red. In general, $R$ could contain further boxes below and left of the bottom of $S$.
\end{example}
\begin{lemma} \label{squish}
Let $T \in \PSVT(\lambda,n)$ be a $P$-shifted set-valued tableau and let $\Delta$ be the largest D-partition contained in $\lambda$. Then there exists a tableau $S \in \PSVT(\Delta,n)$ such that $d(T) = d(S)$.
\end{lemma}
\begin{proof}
If $\lambda = \Delta$, then there is nothing to show. Otherwise, $\lambda$ is not a D-partition, so there exists at least one row $k$ for which $\lambda_k-1 = \lambda_{k+1}$; choose the largest such $k$, and let $\mathsf{R}_0$ be the rightmost box in row $k$. By the choice of $k$, $\mathsf{R}_0^\uparrow$ exists. Consider the longest possible sequence of boxes
\[\mathsf{R}_0,\mathsf{U}_1,\mathsf{R}_1,\mathsf{U}_2, \mathsf{R}_2, \dots, \mathsf{U}_j, (\mathsf{R}_j)\]
such that $\mathsf{U}_i = \mathsf{R}_{i-1}^\uparrow$ and $\mathsf{R}_i = \mathsf{U}_i^\leftarrow$ (where the final box $\mathsf{R}_j$ may or may not exist). This sequence of boxes forms a \emph{short ribbon} in the sense of \cite{Thomas.Yong} and contains at least two boxes. Two nominally distinct cases can occur, based on whether the final box in the ribbon is an $\mathsf{R}$ or a $\mathsf{U}$:
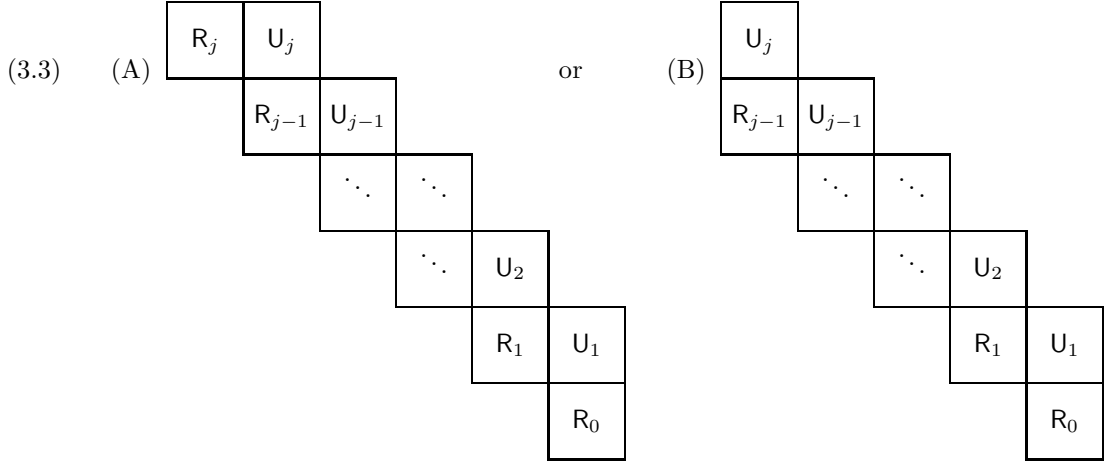
\begin{equation}
(\textup{A}) \hspace{.2cm}
\ytableausetup{boxsize=1cm}
\begin{ytableau}
\mathsf{R}_j & \mathsf{U}_j\\
\none & \mathsf{R}_{j-1} & \mathsf{U}_{j-1}\\
\none & \none & \ddots & \ddots\\
\none & \none & \none & \ddots & \mathsf{U}_2\\
\none & \none & \none & \none & \mathsf{R}_1 & \mathsf{U}_1 \\
\none & \none & \none & \none & \none & \mathsf{R}_0
\end{ytableau} 
\hspace{-1cm} \textup{ or } \hspace{1cm} (\textup{B}) \hspace{.2cm}
\begin{ytableau}
\mathsf{U}_j\\
\mathsf{R}_{j-1} & \mathsf{U}_{j-1}\\
\none & \ddots & \ddots\\
\none & \none & \ddots & \mathsf{U}_2\\
\none & \none & \none & \mathsf{R}_1 & \mathsf{U}_1 \\
\none & \none & \none & \none & \mathsf{R}_0
\end{ytableau}
\end{equation}
Observe that, by choice of $k$, the boxes $\mathsf{U}_i^\uparrow$ and $\mathsf{U}_i^\rightarrow$ do not exist for any $i$. In case (A), where $\mathsf{R}_j$ exists, the box $\mathsf{R}_j^\uparrow$ does not exist. In case (B), the box $\mathsf{U}_j^\leftarrow$ does not exist.

We define a sequence of fillings (not necessarily $P$-shifted set-valued tableaux) recursively as follows. Let $\lambda^{(0)} = \lambda$ and $T^{(0)} = T$. For $i > 0$, let $\lambda^{(i)} = \lambda^{(i-1)} \setminus \mathsf{U}_i$. Take $T^{(i)}$ to be the filling of $\lambda^{(i)}$ defined by:
\begin{itemize}
\item{$T^{(i)}(\mathsf{R}_{i-1}) = T^{(i-1)}(\mathsf{R}_{i-1}) \cup T^{(i-1)}(\mathsf{U}_i)$},
\item{$T^{(i)}(\mathsf{R}_i) = T^{(i-1)}(\mathsf{R}_i) \cup \bigg( T^{(i-1)}(\mathsf{U}_i) \cap T^{(i-1)}(\mathsf{R}_{i-1}) \bigg)$},
\item{$T^{(i)}(\mathsf{B}) = T^{(i-1)}(\mathsf{B})$, for all other boxes $\mathsf{B}$.}
\end{itemize}
That is to say, at the $i$th step of the recursion, we delete the box $\mathsf{U}_i$ from the diagram, slide the contents of $\mathsf{U}_i$ down into $\mathsf{R}_{i-1}$, and to preserve degree, place any intersection between $T^{(i-1)}(\mathsf{U}_i)$ and $T^{(i-1)}(\mathsf{R}_{i-1})$ into $\mathsf{R}_i$. (One might be concerned that this description requires, in case (A), placing labels in $\mathsf{R}_j$ when that box does not exist; however, we will show that this does not occur.)
For example, the local configuration below changes in the following way:
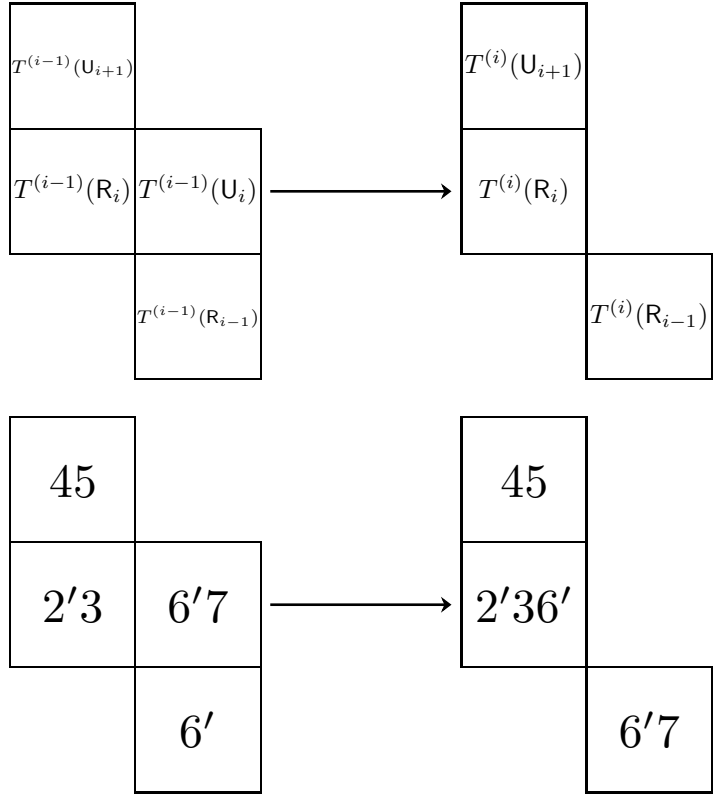
\[
\begin{tikzpicture}
\ytableausetup{boxsize=1.65cm}

\node (a) at (0,0) {
\begin{ytableau}
{\scriptstyle T^{(i-1)}(\mathsf{U}_{i+1})} \\
T^{(i-1)}(\mathsf{R}_i) & T^{(i-1)}(\mathsf{U}_i)\\
\none & {\scriptstyle T^{(i-1)}(\mathsf{R}_{i-1})}
\end{ytableau}
};

\node (b) at (6,0) {
\begin{ytableau}
T^{(i)}(\mathsf{U}_{i+1}) \\
T^{(i)}(\mathsf{R}_i)\\
\none & T^{(i)}(\mathsf{R}_{i-1})
\end{ytableau}
}; 

\node (c) at (0,-5.5) {
\begin{ytableau}
\raisebox{-.4em}{\scalebox{1.8}{$45$}} \\
\raisebox{-.4em}{\scalebox{1.8}{$2'3$}} & \raisebox{-.4em}{\scalebox{1.8}{$6'7$}}\\
\none & \raisebox{-.4em}{\scalebox{1.8}{$6'$}}
\end{ytableau}
};

\node (d) at (6,-5.5) {
\begin{ytableau}
\raisebox{-.4em}{\scalebox{1.8}{$45$}} \\
\raisebox{-.4em}{{\scalebox{1.8}{$2'36'$}}}\\
\none & \raisebox{-.4em}{\scalebox{1.8}{$6'7$}}
\end{ytableau}
};

\draw[-{Stealth[width=6pt,length=4pt]}, line width=1pt](a) -- (b);
\draw[-{Stealth[width=6pt,length=4pt]}, line width=1pt](c) -- (d);
\end{tikzpicture}
\]
We make the following claims about the sequence $(\lambda^{(i)},T^{(i)})_{i=0}^j$.

\begin{claim}
The sequence $(\lambda^{(i)},T^{(i)})$ satisfies  
	\begin{enumerate}
		\item $\Delta \subseteq \lambda^{(j)} \subset \lambda^{(j-1)} \subset \dots \subset \lambda^{(1)} \subset \lambda^{(0)}$;
		\item if $T^{(i)}$ is not a $P$-shifted set-valued tableau, then the only violation is an increasingness violation that occurs between the boxes $\mathsf{R}_i$ and $\mathsf{U}_{i+1}$; 
		\item  $T^{(i-1)}(\mathsf{U}_i) \cap T^{(i-1)}(\mathsf{R}_{i-1})$ consists only of primed letters; and
\item $d(T^{(i)}) = d(T^{(i-1)})$.  
	\end{enumerate}
\end{claim}
\begin{proof}[Proof of the claim]
We prove the four statements in turn.
\begin{enumerate}
\item Each step from $\lambda^{(i)}$ to $\lambda^{(i+1)}$ replaces two consecutive rows of the partition $(r,r-1)$ with $(r,r-2)$; this cannot change the largest D-partition contained within the partition.
\item This follows inductively, the base case $i = 0$ being true by the assumption that $T^{(0)} = T$ is a $P$-shifted set-valued tableau. If $T^{(i-1)}$ is a valid tableau except for $\mathsf{U}_i$, then by removing the box $\mathsf{U}_i$ and only changing the filling by adding elements to the boxes $\mathsf{R}_i$ and $\mathsf{R}_{i-1}$ which are larger than all of the current contents, we can only introduce an increasingness violation involving these two boxes and a box above or to the right. $\mathsf{R}_{i-1}$ no longer has any boxes above or to the right in $\lambda^{(i)}$, and $\mathsf{R}_i$ has nothing to its right, so the only possible violation is between $\mathsf{R}_i$ and $\mathsf{U}_{i+1}$, as claimed.
\item This also follows by induction, the base case $i = 1$ following from the fact that $T$ is a valid tableau so the intersection in question is either empty or a single primed letter. By definition, $T^{(i-1)}(\mathsf{R}_{i-1})$ is $T(\mathsf{R}_{i-1}) \cup \bigg( T^{(i-2)}(\mathsf{U}_{i-1}) \cap T^{(i-2)}(\mathsf{R}_{i-2}) \bigg) $, which is by induction $T(\mathsf{R}_{i-1})$ together with (perhaps) some primed letters. We have $T^{i-1}(\mathsf{U}_i) = T(U_i)$, and since $T$ is a $P$-shifted set-valued tableau, the intersection $T(\mathsf{U}_i) \cap T(\mathsf{R}_{i-1})$ is either empty or a single primed letter. This proves this part of the claim.
\item If $T^{(i-1)}(\mathsf{U}_i) \cap T^{(i-1)}(\mathsf{R}_{i-1}) = \varnothing$, then this is trivial. Otherwise, $T^{(i-1)}(\mathsf{U}_{i}) \cap T^{(i-1)}(\mathsf{R}_{i-1})$ is a set of primed letters by Claim (3), and because the original filling $T$ is row-increasing, none of these primed letters are contained in $T^{(i-1)}(\mathsf{R}_i)$, so 
\[
|T^{(i)}(\mathsf{R}_i)| = |T^{(i-1)}(\mathsf{R}_i)| + |T^{(i-1)}(\mathsf{U}_i) \cap T^{(i-1)}(\mathsf{R}_{i-1})|
\]
to compensate for the fact that 
\[
|T^{(i)}(\mathsf{R}_{i-1})| = |T^{(i-1)}(\mathsf{R}_{i-1}|+|T^{(i-1)}(\mathsf{U}_i)|-|T^{(i-1)}(\mathsf{U}_i) \cap T^{(i-1)}(\mathsf{R}_{i-1})|,
\]
and the degrees balance. 
\end{enumerate}
This completes the proof of the claim.
\end{proof}

 Now, we deal with the concern that case (B) appears to define a filling for a box $\mathsf{R}_j$ that does not exist. However, it can be seen that, in case (B), the definition of $T^{(j)}(\mathsf{R}_j)$ will always be empty; since $\mathsf{U}_j$ is on the main diagonal, $T^{(j-1)}(\mathsf{U}_j) = T(\mathsf{U}_j)$ contains no primed entries, and so the intersection $T^{(j-1)}(\mathsf{U}_j) \cap T^{(j-1)}(\mathsf{R}_{j-1})$ is empty. Since $T(\mathsf{R}_j)$ also originally sits empty, the $T^{(j)}(\mathsf{R}_j)$ prescribed by the algorithm is empty, and this procedure is in fact well-defined.

It then follows from Claim (2) that $T^{(j)}$ is a $P$-shifted set-valued tableau on $\lambda^{(j)}$ (since $\mathsf{U}_{j+1}$ does not exist), from Claim (1) that $\Delta \subseteq \lambda^{(j)} \subset \lambda$, and from Claim (4) that $d(T^{(j)}) = d(T)$, and so the $P$-shifted set-valued tableau $T^{(j)}$ on the strict partition $\lambda^{(j)}$ suffices to establish the lemma.
\end{proof}

\begin{example} \label{squishexample}
Below, we trace through the algorithm of Lemma~\ref{squish} applied to a tableau of shape $7642$ to reduce to a tableau of D-partition shape $7531$. Boxes denoted $\mathsf{U}_i$ are shaded in red \textcolor{red}{$\blacksquare$}, while boxes denoted $\mathsf{R}_i$ are shaded in blue \textcolor{cyan}{$\blacksquare$}. 

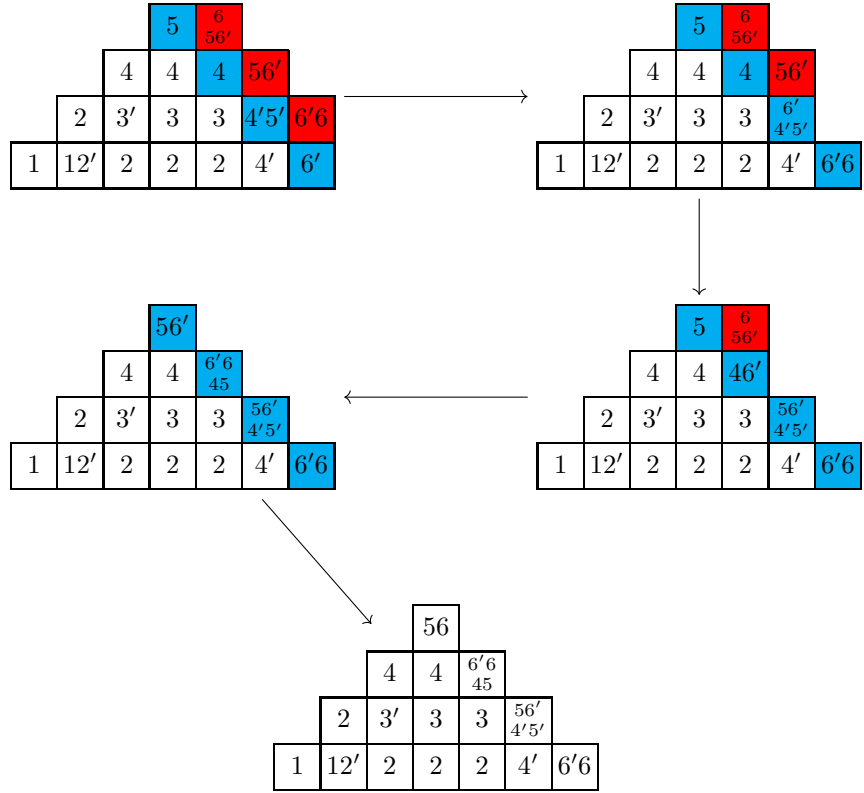
\[
\ytableausetup{boxsize=.6cm}
\begin{tikzpicture}
\node (a) at (0,0) {
\begin{ytableau} 
\none & \none & \none & *(cyan) 5 & *(red) \genfrac{}{}{0pt}{}{6}{56'}\\
\none & \none & 4 & 4 & *(cyan) 4 & *(red) 56'\\
\none & 2 & 3' & 3 & 3 & *(cyan) 4'5' & *(red) 6'6\\
1 & 12' &  2 & 2 & 2 & 4' & *(cyan) 6'
\end{ytableau}
};
\node (b) at (7,0) {
\begin{ytableau} 
\none & \none & \none & *(cyan) 5 & *(red) \genfrac{}{}{0pt}{}{6}{56'}\\
\none & \none & 4 & 4 & *(cyan) 4 & *(red) 56'\\
\none & 2 & 3' & 3 & 3 & *(cyan)\genfrac{}{}{0pt}{}{6'}{4'5'}\\
1 & 12' &  2 & 2 & 2 & 4' & *(cyan) 6'6
\end{ytableau}
};
\node (c) at (7,-4) {
\begin{ytableau} 
\none & \none & \none & *(cyan) 5 & *(red) \genfrac{}{}{0pt}{}{6}{56'}\\
\none & \none & 4 & 4 & *(cyan) 46'\\
\none & 2 & 3' & 3 & 3 & *(cyan) \genfrac{}{}{0pt}{}{56'}{4'5'}\\
1 & 12' &  2 & 2 & 2 & 4' & *(cyan) 6'6
\end{ytableau}
};
\node (d) at (0,-4) {
\begin{ytableau} 
\none & \none & \none & *(cyan) 56'\\
\none & \none & 4 & 4 & *(cyan) \genfrac{}{}{0pt}{}{6'6}{45}\\
\none & 2 & 3' & 3 & 3 & *(cyan) \genfrac{}{}{0pt}{}{56'}{4'5'}\\
1 & 12' &  2 & 2 & 2 & 4' & *(cyan) 6'6
\end{ytableau}
};
\node (e) at (3.5,-8) {
\begin{ytableau} 
\none & \none & \none & 56\\
\none & \none & 4 & 4 & \genfrac{}{}{0pt}{}{6'6}{45}\\
\none & 2 & 3' & 3 & 3 & \genfrac{}{}{0pt}{}{56'}{4'5'}\\
1 & 12' &  2 & 2 & 2 & 4' & 6'6
\end{ytableau}
};

\draw[->] (a) -- node[midway,right] {} (b);
\draw[->] (b) -- node[midway,right] {} (c);
\draw[->] (c) -- node[midway,right] {} (d);
\draw[shorten > = -.5cm,->] (d) -- node[midway,right] {} (e);
\end{tikzpicture}
\]
\end{example}

\begin{definition} \label{maxtableau}
Let $\Delta$ be a D-partition of length $\ell$ and fix a natural number $n \geq \ell$. If $\Delta_\ell > 1$, we define the $P$-shifted set-valued tableau $\mathcal{M}_{\Delta,n}$ to be
\begin{equation}
\ytableausetup{boxsize=1cm}
\begin{ytableau}
\none & \none & \none & \none & \ell & \ell & \dots & [\ell,n]_\mathbb{S}\\
\none & \none & \none & \iddots & \vdots & \vdots & \vdots & \vdots & \ddots \\
\none & \none & \iddots & \iddots & \vdots & \vdots & \vdots & \vdots & \ddots & \ddots \\
\none 2 & 2 & 2 & 2 & 2 & 2 & 2 & 2 & 2 & \dots & [2,n]_\mathbb{S}\\
1 & 1 & 1 & 1 & 1 & 1 & 1 & 1 & 1 & \dots & 1 & [1,n]_\mathbb{S}
\end{ytableau}.
\end{equation}
That is, each box in row $i$ is filled with the value $i$, except the rightmost box in row $i$ which receives the interval $[i,n]_\mathbb{S}$.
If instead $\Delta_\ell = 1$, we define $\mathcal{M}_{\Delta,n}$ to be
\begin{equation}
\begin{ytableau}
\none & \none & \none & \none & \none & [\ell,n]\\
\none & \none & \none & \none & k & \cdots & k & [k,n]_\mathbb{S} \\
\none & \none & \none & \vdots & \vdots & \vdots & \vdots & \vdots & \ddots \\
\none & \none & \vdots & \vdots & \vdots & \vdots & \vdots & \vdots & \vdots & \ddots \\
\none 2 & 2 & 2 & 2 & 2 & 2 & 2 & 2 & 2 & \dots & [2,n]_\mathbb{S}\\
1 & 1 & 1 & 1 & 1 & 1 & 1 & 1 & 1 & \dots & 1 & [1,n]_\mathbb{S}
\end{ytableau},
\end{equation}
where $k = \ell - 1$.
That is, each box outside the top row is filled as before, while the unique box of the top row receives the integer interval $[\ell,n]$.
\end{definition}

\begin{lemma} \label{push}
For any D-partition $\Delta$ and any $n \in \mathbb{N}$, the tableau $\mathcal{M}_{\Delta,n}$ has maximum degree among all tableaux in $\PSVT(\Delta,n)$.
\end{lemma}
\begin{proof}
For an arbitrary tableau $T \in \PSVT(\Delta,n)$, we will construct a finite sequence of tableaux $T = T_0,T_1 \dots, T_{j-1}, T_j = \mathcal{M}_{\Delta,n}$ such that $d(T_0) \leq d(T_1) \leq \dots \leq d(T_j)$.
If $T = \mathcal{M}_{\Delta,n}$, there is nothing to show, so we assume there exists at least one box $\mathsf{B} \in \Delta$ for which $T(\mathsf{B}) \not= \mathcal{M}_{\Delta,n}(\mathsf{B})$. Find the smallest index $i$ for which a box in row $i$ of $\Delta$ differs between $T$ and $\mathcal{M}_{\Delta,n}$, and find the leftmost box $\mathsf{B}_{\bad}$ in this row such that  $T(\mathsf{B}_{\bad}) \not= \mathcal{M}_{\Delta,n}(\mathsf{B}_{\bad})$. There are three nominally distinct cases to treat:

\bigskip
\noindent
{\sf (Case 1: $\mathsf{B}_{\bad}$ is not the rightmost box in row $i$):}
In this case, we define the tableau $T_1$ as follows:
\begin{itemize}
\item{$T_1(\mathsf{B}_{\bad}) = \{i\}$,
}
\item{$T_1(\mathsf{B}_{\bad}^\rightarrow) = T(\mathsf{B}_{\bad}) \cup T(\mathsf{B}_{\bad}^\rightarrow)$,
}
\item{$T_1(\mathsf{B}) = T(\mathsf{B})$, for all other boxes $\mathsf{B}$.}
\end{itemize}
The only two boxes which are different in $T_1$ and $T$ are $\mathsf{B}_{\bad}$ and $\mathsf{B}_{\bad}^\rightarrow$, so to confirm that $T_1$ remains a valid tableau we only need to check the following local region. 
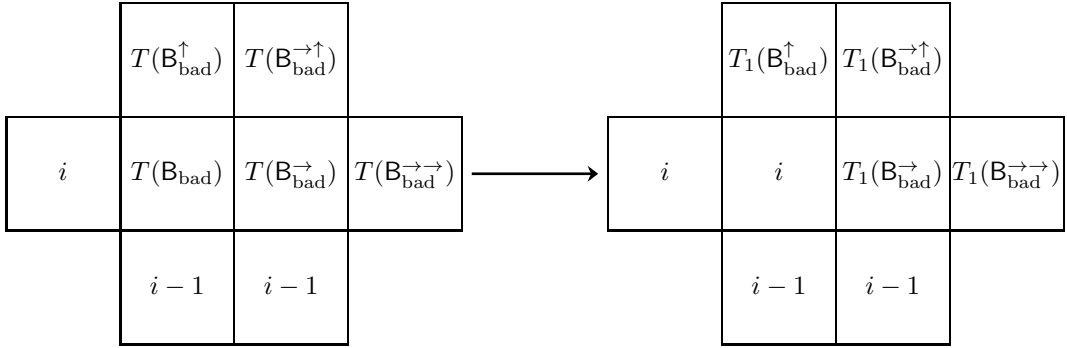
\[
\begin{tikzpicture}
\ytableausetup{boxsize=1.5cm}

\node (m) at (0,0) {\begin{ytableau}
\none & T(\mathsf{B}_{\bad}^\uparrow) & T(\mathsf{B}_{\bad}^{\rightarrow\uparrow}) & \none\\
i & T(\mathsf{B}_{\bad}) & T(\mathsf{B}_{\bad}^\rightarrow) & T(\mathsf{B}_{\bad}^{\rightarrow\rightarrow})\\
\none & i-1 & i-1
\end{ytableau}};
\node (f2m) at (8,0) {\begin{ytableau}
\none & T_1(\mathsf{B}_{\bad}^\uparrow) & T_1(\mathsf{B}_{\bad}^{\rightarrow\uparrow}) & \none\\
i & i & T_1(\mathsf{B}_{\bad}^\rightarrow) & T_1(\mathsf{B}_{\bad}^{\rightarrow\rightarrow})\\
\none & i-1 & i-1
\end{ytableau}};
\draw[-{Stealth[width=6pt,length=4pt]}, line width=1pt](m) -- (f2m);\end{tikzpicture}
\]
Most increasingness conditions follow immediately from inspection and $T \in \PSVT(\Delta,n)$; we only need to remark that since we only add elements (weakly) less than the minimum of $T(\mathsf{B}_{\bad}^\rightarrow)$ to $\mathsf{B}_{\bad}^\rightarrow$, we cannot introduce a violation with the boxes $\mathsf{B}_{\bad}^{\rightarrow \uparrow}$ or $\mathsf{B}_{\bad}^{\rightarrow\rightarrow}$. To conclude that $d(T) \leq d(T_1)$, observe that $|T(\mathsf{B}_{\bad}) \cap T(\mathsf{B}_{\bad}^\rightarrow)| \leq 1$, and so
\[|\{i\}| + |T(\mathsf{B}_{\bad}) \cup T(\mathsf{B}_{\bad}^\rightarrow)| \geq 1+(|T(\mathsf{B}_{\bad})|+|T(\mathsf{B}_{\bad}^\rightarrow)|-1).\]
Thus, the total content of these two boxes in $T_1$ is weakly larger than in $T$.

\bigskip
\noindent
{\sf (Case 2: $\mathsf{B}_{\bad}$ is the rightmost box in row $i$ and does not lie on the main diagonal):}
 In this case, we set $T_1(\mathsf{B}_{\bad}) = [i,n]_\mathbb{S}$ and $T_1(\mathsf{B}) = T(\mathsf{B})$ for all other boxes $\mathsf{B}$. Since $\mathsf{B}_{\bad}$ is in row $i$, $\min(T(\mathsf{B}_{\bad})) \geq i$, so $T(\mathsf{B}_{\bad}) \subseteq [i,n]_\mathbb{S}$, and therefore $d(T) \leq d(T_1)$. The tableau $T_1$ is still a valid tableau, because by assumption $T_1(\mathsf{B}_{\bad}^\leftarrow) = \{i\}$ and $T_1(\mathsf{B}_{\bad}^\downarrow) = \{i-1\}$ (or is empty, if $i = 1$), and boxes in the other two directions do not exist. Since $T(\mathsf{B}_{\bad}) \subseteq [i,n]_\mathbb{S}$, it follows that $|T(\mathsf{B}_{\bad})| \leq |[i,n]_\mathbb{S}|$, and since this is the only box which changes between $T$ and $T_1$, we conclude $d(T) \leq d(T_1)$.

\bigskip
\noindent
{\sf (Case 3: $\mathsf{B}_{\bad}$ is the rightmost box in row $i$ and lies on the main diagonal):}
This can only happen if row $i$ consists of only the single box $\mathsf{B}_{\bad}$. In this case, we proceed identically as {\sf Case 2}, except that we are prohibited from having primed entries in $\mathsf{B}_{\bad}$, so we set $T_1(\mathsf{B}_{\bad}) = [i,n]$ instead, and as in {\sf Case 2}, we see that $T_1$ remains a valid tableau with $d(T) \leq d(T_1)$.

\bigskip
If $T_1 \not= \mathcal{M}_{\Delta,n}$, then we can produce another tableau $T_2$ by applying the same construction to $T_1$, and continue to produce a sequence of $P$-shifted set-valued tableaux $T_0,T_1,T_2,\dots$ with the property that $d(T_i) \leq d(T_{i+1})$ for all $i$. This sequence must be finite, because by construction the bad box of $T_{i+1}$ is either in the same row as the bad box of $T_i$ but strictly further right, or in a strictly higher row. The sequence will stop when we are unable to find any box of $T_j$ differing from $\mathcal{M}_{\Delta,n}$, that is, when $T_j = \mathcal{M}_{\Delta,n}$, which completes the proof of the lemma.
\end{proof}

Applying Lemmas~\ref{squish} and~\ref{push} reduces the computation of the degree of $\IN_{\lambda,n}$ to directly counting how many entries are in $\mathcal{M}_{\Delta,n}$, so we can now finish the proof.

\begin{proof}[Proof of Theorem 1.1]
Now, we complete the proof of Theorem~\ref{main} by combining the previous lemmas. Let $\Delta$ be the largest D-partition contained in $\lambda$ and let the last nonzero part of $\Delta$ be $\Delta_\ell$. By Lemma~\ref{lem:order_containment}, $d(\IN_{\Delta,n}) \leq d(\IN_{\lambda,n})$. By Lemma~\ref{squish}, we have $d(\IN_{\Delta,n}) \geq d(\IN_{\lambda,n})$. Therefore, $d(\IN_{\Delta,n}) = d(\IN_{\lambda,n})$. By Lemma \ref{push},  $d(\IN_{\Delta,n}) = d(\mathcal{M}_{\Delta,n})$. If $\Delta_\ell \not= 1$, then by direct inspection of the tableau $\mathcal{M}_{\Delta,n}$, we have
\[c(\mathcal{M}_{\Delta,n}) = (\Delta_1,\Delta_2+2,\dots,\Delta_\ell+2\ell,2\ell,2\ell,\dots,2\ell),\]
so that
\[d(\mathcal{M}_{\Delta,n}) = |\Delta|+\sum_{i = 1}^\ell 2i+2\ell(n-\ell) =  |\Delta| + 2n\ell - \ell^2-\ell.\]
Otherwise, $\Delta_\ell = 1$ and again by inspection we find
\[c(\mathcal{M}_{\Delta,n}) = (\Delta_1,\Delta_2+2,\dots,\Delta_{\ell-1}+2(\ell-1),\Delta_\ell+2\ell-2,2\ell-1,2\ell-1,\dots,2\ell-1),\]
so that
\[d(\mathcal{M}_{\Delta,n}) = |\Delta| + 2n\ell-\ell^2-n. \qedhere\]
\end{proof}

\begin{remark}
    Lemma~\ref{push} explicitly identifies an element of $\PSVT(\lambda,n)$ of maximum degree, when $\lambda$ is a D-partition. Through Lemmas~\ref{lem:order_containment} and~\ref{squish}, this is sufficient to allow us to determine the degree of $\IN_{\lambda,n}$ for general $\lambda$, but without identifying explicit elements of top degree outside of the D-partition case. It would be very interesting to find and describe such representatives.
\end{remark}

\section{Regularity of Pfaffian ideals} \label{regularity}
In this section, we relate the combinatorics of the previous section to commutative algebra. First, in Section~\ref{sec:MSV}, we recall the definition of matrix Schubert varieties, for which the combinatorics developed in Section~\ref{symdegrees} will provide a new regularity formula in the Grassmannian case, complementing those of \cite{RRR,Hafner,RRW,PSW}. Then, in Section~\ref{sec:ssMSV}, we recall skew-symmetric matrix Schubert varieties, and observe that they are Cohen--Macaulay. Finally, in Section~\ref{symplectic.grothendieck}, we recall symplectic Grothendieck polynomials and use them in combination with Theorem~\ref{main} to establish our main algebraic result, Theorem~\ref{regularity_theorem}.

\subsection{Matrix Schubert varieties}\label{sec:MSV}
For textbook treatments of the material in this section, see \cite{Fulton:book,Miller.Sturmfels}.

For a permutation $w \in \symm_n$, the \emph{permutation matrix} $P^w \in \Mat_n$ associated to $w$ is defined as the $n \times n$ matrix
\[P^w_{i,j} = \begin{cases}1, \text{ if } j = w(i);\\ 0, \text{ otherwise.}\end{cases}\]
For example, the permutation matrix $P^{52134}$ is
\[\begin{pmatrix}
0 & 0 & 0 & 0 & 1 \\
0 & 1 & 0 & 0 & 0 \\
1 & 0 & 0 & 0 & 0 \\
0 & 0 & 1 & 0 & 0 \\
0 & 0 & 0 & 1 & 0
\end{pmatrix}.\]
For any matrix $A \in \Mat_n$ and any subsets $I,J \subseteq [n]$,  we define $A_{IJ}$ to be the submatrix \[
\{A_{i,j}:(i,j) \in I \times J\}.
\]
In particular, the matrix $A_{[i][j]}$ is the principal $i \times j$ minor of $A$.

The \emph{rank matrix} $R^w$ of $w$ is the matrix defined by the condition that $R^w_{i,j}$ is the rank of the $i \times j$ principal minor $P^w_{[i][j]}$. For example, the rank matrix $R^{52134}$ is
\[\begin{pmatrix}
0 & 0 & 0 & 0 & 1 \\
0 & 1 & 1 & 1 & 2 \\
1 & 2 & 2 & 2 & 3 \\
1 & 2 & 3 & 3 & 4 \\
1 & 2 & 3 & 4 & 5
\end{pmatrix}.\]

The \emph{matrix Schubert variety} $X_w$ is the set of matrices 
\[X_w \coloneqq \{A \in \Mat_n: \rank A_{[i][j]} \leq R^w_{i,j}\},\]
which is indeed an affine algebraic variety since each inequality is equivalent to the polynomial condition that all $\left(R^w_{i,j}+1\right) \times \left(R^w_{i,j}+1\right)$ minors of $A_{[i][j]}$ vanish. These varieties, first introduced by W.~Fulton \cite{Fulton}, provide an affine model of for Schubert varieties, and have been of significant interest for many years (see, e.g., \cite{Escobar.Meszaros,Fink.Rajchgot.Sullivant,Hamaker.Pechenik.Weigandt,Hsiao,Klein.Weigandt,Knutson.Miller,Knutson.Miller.Yong,PSW}). Fulton \cite{Fulton} shows that the defining determinantal conditions generate a prime ideal $I_w$ of the ring $S \coloneqq \mathbb{C}[x_{i,j} : 1 \leq i,j \leq n]$. We are interested in the coordinate ring $S/I_w$.

A permutation $w \in \symm_n$ that fixes $n$ determines a permutation $w' \in \symm_{n-1}$. We routinely identify $w$ with $w'$ or write $w = w' \times 1$ for clarity. Everything we study is invariant under the transformation $w \mapsto w \times 1$; in particular, $S/I_w  \cong S/ I_{w'}$. 

A permutation $w \in \symm_n$ is \emph{Grassmannian} if there is at most one value $1 \leq i < n$ such that $w(i) > w(i+1)$. In Section~\ref{symdegrees}, we will be interested in those matrix Schubert varieties $X_w$ with $w$ Grassmannian. The \emph{code} of a permutation $w \in \symm_n$ is $\code(w) = (c_1, \dots, c_n)$, where $c_i$ is the number of integers $j$ with $j> i$ and $w(j) < w(i)$. Sorting the entries of $\code(w)$ into a partition yields the shape $\shape(w)$ of $w$. 

Permutations are uniquely determined by their codes. Given a nonempty partition $\lambda$ and a positive integer $n \geq \ell(\lambda)$, there is a unique Grassmannian permutation $w_\lambda$ with $w(n) > w(n+1)$ and $\shape(w) = \lambda$. To find this permutation, extend $\lambda$ to have length $n$ by appending the needed number of terminal $0$s and then reverse $\lambda$ to obtain a weakly increasing sequence. This sequence is uniquely the code of the desired permutation $w_\lambda$.

\subsection{Skew-symmetric matrix Schubert varieties}\label{sec:ssMSV}
Let $\SSM_{n}$ denote the variety of $n \times n$ skew-symmetric matrices, a linear subspace of $\Mat_n$. Our focus is on the \emph{skew-symmetric matrix Schubert varieties} $\SSX_W \coloneqq X_W \cap \SSM_{n}$ as studied by E.~Marberg and B.~Pawlowski 
in \cite{Marberg.Pawlowski:Grobner}. These varieties provide an affine model of orbit-closures for the action of the symplectic group $\Sp_n(\mathbb{C})$ on the flag variety $\flags_n$.

While $\SSX_w$ is defined for arbitrary permutations, this is not the appropriate generating set, as for example, $\SSX_{21} = \SSX_{12}$, since the diagonal entries of a skew-symmetric matrix are necessarily zero. Instead, in the symplectic context the relevant indexing family for cohomology classes is the set of \emph{fixed-point-free involutions}, which is the set
\[ \FPF_n \coloneqq \{z \in \symm_{n}: z^2 = 1 \text{ and } z(i) \not= i \text{ for } 1 \leq i \leq n\}.\]
Note this set is empty if $n$ is odd, so from this point forward we assume we are working with $2n \times 2n$ matrices and the symmetric group $\symm_{2n}$.

A second, subtler issue which arises is that the ideal of minors defining $\SSX_z$ is not generally prime. This differs in a significant way from the generic matrix setting, where it is established by \cite{Fulton} that the ideal of minors defining $X_w$ is always prime. Constructing the ideal $I(\SSX_z)$ requires some intricacies with Pfaffian polynomials which we cover below, closely following \cite{Marberg.Pawlowski:Grobner}, where one can look for full proofs and additional information.

Recall that a skew-symmetric matrix $A \in \SSM_{2n}$ carries an invariant polynomial known as the \emph{Pfaffian} $\pf(A)$ with the property that $\pf(A)^2 = \det(A)$. Formally, we define it as
\[\pf(A) \coloneqq \frac{1}{2^n n!}\sum_{\sigma \in \symm_{2n}}\text{sgn}(\sigma)\prod_{i = 1}^n a_{\sigma(2i-1),\sigma(2i)}.\]
For background on combinatorial appearances of Pfaffians, see \cite{Godsil}.

Fix $n$ and let $S = \mathbb{C}[x_{i,j} : 1 \leq j < i \leq 2n]$ be a polynomial ring in $\binom{2n}{2}$ independent indeterminates.

\begin{theorem} \cite{Marberg.Pawlowski:Grobner}
Let $z \in \FPF_{2n}$, and let $\mathcal{X}^{\mathrm{ss}}$ be the $2n \times 2n$ skew-symmetric matrix
\[\begin{pmatrix}
0 & -x_{2,1} & \dots & -x_{2n,1} \\
x_{2,1} & 0 & \dots & -x_{2n,2} \\
\vdots & \ddots & \ddots & \vdots \\
x_{2n,1} & \dots & \dots & 0
\end{pmatrix}.
\]
Then the radical ideal $I(\SSX_z) \subset S$ is
\[I(\SSX_z) = \langle \pf(\mathcal{X}_{UU}^{\mathrm{ss}}):\exists (i,j) \in (2n \times 2n), i \geq j, U \subseteq [i],|U \cap [j]| > R^z_{i,j} \rangle
.\]
\end{theorem}

\begin{theorem}\label{thm:CM}
	The coordinate ring $S / I(\SSX_z)$ of the variety $\SSX_z$ is Cohen--Macaulay.
\end{theorem}
\begin{proof}
	 Marberg--Pawlowski \cite{Marberg.Pawlowski:Grobner} show that there is a term order such than an initial ideal of $I(\SSX_z)$ is squarefree and that the corresponding Stanley--Reisner simplicial complex is shellable, so that for this term order $\textrm{in}(I(\SSX_z))$ is Cohen--Macauley (see, e.g., \cite[Theorem~5.13]{Ene.Herzog}). Thus it follows (see, e.g., \cite[Corollary~6.9]{Ene.Herzog}) that $S/I(\SSX_z)$ is Cohen--Macaulay.
\end{proof}

\begin{remark}
    E.~De~Negri and E.~Gorla \cite{DeNegri.Gorla:liaison} show that \emph{mixed ladder Pfaffian varieties} are arithmetically Cohen--Macaulay. We believe this is related to the vexillary case of Theorem~\ref{thm:CM}, but do not currently understand the the precise connection; forthcoming work of L.~Escobar, A.~Fink, J.~Rajchgot, and A.~Woo \cite{Escobar.Fink.Rajchgot.Woo} is expected to shed more light on this.
\end{remark}

\subsection{Symplectic Grothendieck polynomials and proof of Theorem~\ref{regularity_theorem}} \label{symplectic.grothendieck}

    Let $\partial_i$ be the divided difference operator $\partial_i f \coloneqq \frac{f-s_if}{x_i-x_{i+1}}$, and let $\kpartial_i \coloneqq \partial_i(1 - x_{i+1})f)$. 

The symplectic Grothendieck polynomials are, like ordinary Schubert polynomials, constructed from a top element by sequences of divided difference operators. They are defined as follows.

\begin{definition-theorem}[{\cite{Wyser.Yong, Marberg.Pawlowski:properties}}] \label{symplectic-polynomials} 
The \emph{symplectic Grothendieck polynomials} in $2n$ variables are the unique family of polynomials $\{\grothsp_z\}_{z \in \FPF_{2n}}$ satisfying
\begin{enumerate}
\item $\grothsp_{2n(2n-1)\dots321} = \prod_{1 \leq i < j \leq 2n-i}(x_i+x_j - x_ix_j)$; and
\item if $i+1 \not= z(i)$ and $i \not= z(i+1)$ and $z(i) > z(i+1)$, then $\grothsp_{s_izs_i} = \kpartial_i \grothsp_z$.
\end{enumerate}
\end{definition-theorem}

An involution $z \in \FPF_n$ has a \emph{symplectic code} \cite[\S 4.4]{Marberg.Pawlowski:properties} $\Spcode(z) = (c_1, \dots, c_n)$ where $c_i$ is the number of integers $j$ with 
\[z(i) > z(j) < i < j.\] Sorting the entries of $\Spcode(z)$ into a partition and then taking the transpose yields the \emph{symplectic shape} $\Spshape(z)$ of $z$ \cite[\S 4.4]{Marberg.Pawlowski:properties}.

It is shown in \cite[\S 4.2]{Marberg.Pawlowski:formulas} and \cite[Theorem~1.9]{Marberg:refinement}, that the stable limit 
\[\lim_{n \to \infty} \grothsp_{(21)^n \times z}\] 
exists (in a ring of power series) and is a finite sum of $\IN_\lambda$s. We say that $z \in \FPF_n$ is \emph{FPF-vexillary} if this stable limit is a single $\IN_\lambda$. (This differs slightly from the more usual definition, but is equivalent by \cite[Theorem~3.23]{Marberg.Scrimshaw}.) An explicit characterization of FPF-vexillary involutions in terms of pattern avoidance appears in \cite[Corollary~7.9]{Hamaker.Marberg.Pawlowski:JComb}; we do not recall it here as it is somewhat complicated and the details will not play a role in this paper.



If $z$ is FPF-vexillary with symplectic shape $\lambda = \Spshape(z)$, we have by \cite[Theorem~3.23]{Marberg.Scrimshaw}
\begin{equation}\label{eq:stable_limits}
\lim_{n \to \infty} \grothsp_{(21)^n \times z} = \IN_\lambda = \lim_{k \to \infty} \IN_\lambda(x_1, \dots, x_k). 
\end{equation}

We need to understand the nature of this limit in slightly more detail.
Given two polynomials $f,g \in S$ with real coefficients, we say $f \preceq g$ if, for each monomial $m$, we have
\begin{align*}
    [m]f \geq 0 &\Leftrightarrow [m]g \geq 0 \quad \text{and} \\
    |[m]f| &\leq  |[m]g|,
\end{align*}
where $[m]h$ denotes the coefficient on the monomial $m$ in $h$. Note that symplectic Grothendieck polynomials have real (and indeed integral) coefficients.

\begin{lemma}\label{lem:Hecke_atoms}
    For $z \in \FPF_{2n}$, we have $\grothsp_z \preceq \grothsp_{21 \times z}$.
\end{lemma}
\begin{proof}
As defined in \cite[\S 3.3]{Marberg.Pawlowski:formulas}, a \emph{Hecke atom} for $z$ is a permutation $w$ such that $w$ acts on the fixed-point-free involution $2143\ldots(2n)(2n-1)$ according to an action studied in \cite{Rains.Vazirani} to produce $z$. This action is somewhat complicated, so we omit the details; however, it is easy to see from the definition that  
 if $w$ is a Hecke atom for $z$, then $1^2 \times w$ is a Hecke atom for $21 \times z$. We then have that, by \cite[Theorem~3.12]{Marberg.Pawlowski:formulas}, which gives a formula for the symplectic Grothendieck polynomial $\grothsp_u$ in terms of Hecke atoms of $u$, the lemma follows.
\end{proof}

\begin{corollary}\label{cor:poly_comparison}
    If $z \in \FPF_{2n}$ is FPF-vexillary and the last nonzero entry of $\Spcode(z)$ is in position $k$, then 
    \[
    \grothsp_z \preceq \IN_{\Spshape(z),k}.
    \]
    In particular, $\deg \grothsp_z \leq \deg \IN_{\Spshape(z),k}$.
\end{corollary}
\begin{proof}
Since the last nonzero entry of $\Spcode(z)$ is in position $k$, we have
$\grothsp_z \in \mathbb{C}[x_1, \dots, x_k]$ by the involution pipe dream formulas of \cite{Hamaker.Marberg.Pawlowski:pipe}.
    The corollary is then immediate from Lemma~\ref{lem:Hecke_atoms} combined with \eqref{eq:stable_limits}.
\end{proof}


Finally, let us recall the theorem connecting the ideals $I(\SSX_z)$ to our degree calculations:

\begin{theorem}[\cite{Marberg.Pawlowski:Grobner}]\label{thm:MPGrobner}
Let $z \in \FPF_{2n}$. Then
\[\grothsp_z(1-t,1-t,\dots,1-t) = \mathcal{K}(I(\SSX_z); t).\]
\end{theorem}

We may now conclude the main theorem of this section, which is the precise version of Theorem~\ref{regularity_theorem}.
\begin{theorem}
Let $z \in \FPF_{2n}$ be an $FPF$-vexillary fixed-point-free involution with last nonzero entry of $\Spcode(z)$ in position $k$. Further let $\Delta \subseteq \Spshape(z)$ be the largest D-partition contained in $\Spshape(z)$. Then
\[\reg S/I(\SSX_z) \leq \begin{cases}
 	2k\ell-\ell^2-\ell - \big(|\Spshape(z)| - |\Delta|\big), & \text{if } \Delta_\ell > 1; \\
 	2k\ell-\ell^2-k  - \big(|\Spshape(z)| - |\Delta|\big), & \text{if }  \Delta_\ell = 1.
 \end{cases}\]
\end{theorem}
\begin{proof}
    Since $\SSX_z$ is Cohen--Macaulay by Theorem~\ref{thm:CM}, we have that Theorem~\ref{cohenmacreg} applies. Hence $\reg S/I(\SSX_z) = \deg \mathcal{K}(I(\SSX_z); t) - \height I(\SSX_z)$. 
    By Theorem~\ref{thm:MPGrobner}, $\deg \grothsp_z = \deg \mathcal{K}(I(\SSX_z); t).$ 
    We have $\height I(\SSX_z) = |\Spshape(z)|$ (this follows from example from combining \cite[Theorem~1.2]{Hamaker.Marberg.Pawlowski:JComb} with \cite[Theorem~1.2]{Hamaker.Marberg.Pawlowski:pipe}). By Corollary~\ref{cor:poly_comparison}, $\deg \grothsp_z \leq \deg \IN_{\Spshape(z),k}$. From Theorem~\ref{main} we have that 
    \[
    \deg(\IN_{\lambda,k}) = \begin{cases}
 	|\Delta| + 2k\ell-\ell^2-\ell, & \text{if } \Delta_\ell > 1; \\
 	|\Delta|+2k\ell-\ell^2-k, & \text{if }  \Delta_\ell = 1.
 \end{cases}
    \]
    Thus the theorem follows.
\end{proof}

 

\section{Degrees of symmetric Grothendieck polynomials} \label{symdegrees}

In this section, we obtain a formula for the degree of (type A) symmetric Grothendieck polynomials, complementary to that of \cite{RRR}. Our formula differs in appearance from the formula of \cite{RRR}, and neither formula appears to follow directly from the other.  (However, see Proposition~\ref{prop:anna_sez} and surrounding discussion for tentative relations to the more general formula of \cite{PSW}.) Our approach parallels the shifted arguments of the previous sections, with an analogous collection of lemmas.  The proofs of the analogues of Lemmas~\ref{lem:order_containment} and~\ref{squish} are somewhat easier than in the shifted setting because two vertically adjacent boxes in an ordinary set-valued tableau always have disjoint content. The proof of Lemma~\ref{sympush} is nearly verbatim that of Lemma~\ref{push}. 

\begin{lemma}\label{symordercontainment}
Suppose $\mu \subseteq \lambda$ are partitions and that $n \geq \ell(\lambda)$. If $T \in \SVT(\mu, n)$ is a set-valued tableau, then there exists some $T' \in \SVT(\lambda,n)$ with $d(T') \geq d(T)$.
\end{lemma}
\begin{proof}
It suffices to assume that $|\lambda| - |\mu| = 1$. Let the unique box of $\lambda \setminus \mu$ be $\BBB_0$. To define the filling $T'$, we first set $T'(\BBB_0) = n$. 

Let the boxes of the column of $\BBB_0$ be $\BBB_0, \BBB_1, \dots$ from top to bottom.
Next, we define a partial column $\CCC$ as follows. If $n \not \in \BBB_1$, then $\CCC$ is empty. Otherwise, recursively include $\BBB_i$ in $\CCC$ if 
\begin{itemize}
	\item $\BBB_{i-1} \in \CCC$,
	\item $|T(\BBB_{i-1})| = 1$,
	\item  and $n-i+1 \in T(\BBB_{i})$.  
\end{itemize}
If $\BBB_i \in \CCC$, we also refer to it as $\CCC_i$.

We then define the filling $T'$ on the partial column $\CCC$ as
\begin{equation}\label{eq:relabel}
T'(\CCC_j) = \begin{cases}
	T(\CCC_j) \setminus \max T(\CCC_j), & \text{if } |T(\CCC_j)| > 1;\\
	T(\CCC_j) - 1, & \text{if } |T(\CCC_j)| = 1. 
\end{cases}
\end{equation}
Since $n \geq \ell(\lambda)$, Equation~\eqref{eq:relabel} is well-defined.
Finally, if $\mathsf{A} \neq \BBB_0$ and $\mathsf{A} \not \in \CCC$, then we define $T'(\mathsf{A}) = T(\mathsf{A})$. We claim that this tableau $T'$ has the desired properties.

Since we have added an entry to the new box $\BBB_0$ and the filling $T'$ deletes at most one entry from $T$ (the top case of Equation~\eqref{eq:relabel} can occur at most once), it follows that either $d(T') = d(T)$ or $d(T') = d(T)+1$. In particular, $d(T') \geq d(T)$. 

It only remains to check that $T'$ satisfies the necessary increasingness conditions.
First, consider $\BBB_0$. Since $\BBB_0$ is the rightmost box in its row and $n$ is the largest letter of the alphabet, there can be no row increasingness violations with $\BBB_0$. Furthermore, if $n \in T(\CCC_1)$, then by the construction of the filling, $n \not \in T'(\CCC_1)$, so $T'$ has no violation of increasingness between boxes $\BBB_0$ and $\CCC_1$.

Since the partial column $\CCC$ is column-increasing in $T$, by construction it is also column-increasing in $T'$. Similarly, there can be no column-increasingness violation in $T'$ between the bottom box of $\CCC$ and the box directly below it.

Since we have only decreased the values in the boxes in $\CCC$, we cannot have introduced any increasingness violations between a box in $\CCC$ and the box to its right. We then need only check if we introduced a violation between $\CCC$ and boxes to its left. We show by contradiction that this is impossible. Suppose $\CCC_k$ is the smallest $k$ such that $\max(\CCC_k^\leftarrow) > \min(\CCC_k)$. The box $\CCC_k$ must have had only a single entry, since otherwise its content only weakly decreases. Therefore, letting $a = T(\CCC_k) = T'(\CCC_k) + 1$, this implies $a \in \CCC_k^\leftarrow$. Since $T$ was supposed to be column-increasing, this forces $\min(T(\CCC_k^{\uparrow\leftarrow})) > a$. This, however, cannot occur, since $T'(\CCC_{k-1}) = a$, and as $\CCC_{k-1}$ is the box directly right of $\CCC_k^{\uparrow\leftarrow}$, this implies a row-increasingness violation for a smaller $k$ than our chosen minimum, which is a contradiction (if $k = 1$, this argument implies there is a row-increasingness violation in the box $T'(\BBB_0)$, which we already know to be impossible).
\end{proof}

\begin{lemma} \label{symsquish}
Let $T \in \SVT(\lambda,n)$ be a set-valued tableau, and let $\mu$ be the largest strict partition contained in $\lambda$. Then there exists a tableau $S \in \SVT(\mu,n)$ such that $d(T) = d(S)$.
\end{lemma}
\begin{proof}
If $\lambda = \mu$, we are done. Otherwise $\lambda$ is not strict, so there exists at least one row $k$ for which $\lambda_k = \lambda_{k+1}$; choose the largest such $k$, and let $\mathsf{R}$ be the rightmost box in row $k$. Set $\mathsf{U} \coloneqq \mathsf{R}^\uparrow$. We define the filling $T^*$ on $\lambda \setminus \mathsf{U}$ (which is a valid tableau because $\mathsf{U}$ is a corner box, otherwise it would be a higher row $k$ such that $\lambda_k = \lambda_{k+1}$) as:
\begin{itemize}
\item{$T^*(\mathsf{R}) = T(\mathsf{R}) \cup T(\mathsf{U})$},
\item{$T^*(\mathsf{B}) = T(\mathsf{B})$, for all other boxes $\mathsf{B}$.}
\end{itemize}

This is simpler than the shifted case where one must track a short ribbon up the tableau, as the tableau obtained by deleting just this single box already works. The critical difference in this case is that $T(\mathsf{R}) \cap T(\mathsf{U}) = \varnothing$ by the definition of a set-valued tableau, so it is clear that $d(T^*) = d(T)$. Since $R$ was the rightmost box in its row and we have only added strictly larger numbers to it, we have not introduced any increasingness violations with $\mathsf{R}^{\leftarrow}$ or $\mathsf{R}^\downarrow$. Therefore $\mu \subseteq \lambda \setminus \mathsf{U} \subset \lambda$ and $T^*$ is a set-valued tableau on $\lambda \setminus \mathsf{U}$ with $d(T^*) = d(T)$. By iterating this construction until we eventually remove enough corner boxes to reach $\mu$, we establish the lemma by induction.
\end{proof}

\begin{lemma}\label{lem:sym_degree_equality}
   If $\mu$ is the largest strict partition contained in $\lambda$ and $n \geq \ell(\lambda)$, then
   \[
   \deg \IN_{\mu,n} = \deg \IN_{\lambda,n}.
   \]
\end{lemma}
\begin{proof}
    This follows from Lemmas~\ref{symordercontainment}, \ref{symsquish}, and the tableau formula for $P$-Grothendieck polynomials.
\end{proof}

\begin{lemma} \label{sympush}
For any strict partition $\mu$ with length $\ell$ and any $n \in \mathbb{N}$, the tableau
\begin{equation} \label{sym_maxtab}
\ytableausetup{boxsize=.9cm}
\begin{ytableau}
\ell & \dots & [\ell,n]\\
\vdots & \vdots & \vdots & \vdots\\
2 & 2 & 2 & \dots & [2,n]\\
1 & 1 & 1 & \dots & \dots & \dots & [1,n]
\end{ytableau}
\end{equation}
has maximum degree among all tableaux in $\SVT(\mu,n)$.
\end{lemma}
\begin{proof}
Here we can proceed very similarly to the proof of Lemma \ref{push}, with minor modifications. To further mimic the situation already proved, we call the tableau \eqref{sym_maxtab} by the name $\mathcal{N}_{\mu,n}.$ 

If $T \not= \mathcal{N}_{\mu,n}$, we find the smallest index $i$ for which there exists a box $\BBB$ in row $i$ such that $T(\BBB) \neq \mathcal{N}_{\mu,n}(\BBB)$, denote the leftmost such box in the row $\BBB_\bad$, and define a new tableau $T_1 \in \SVT(\mu,n)$ satisfying $d(T) \leq d(T_1)$. We split into cases analogous to cases 1 and 2 in the proof of Lemma~\ref{push}:

\medskip
\noindent
{\sf (Case 1: $\mathsf{B}_{\bad}$ is not the rightmost box in row $i$):}
In this case, we define the tableau $T_1$ as follows:
\begin{itemize}
\item{$T_1(\mathsf{B}_{\bad}) = \{i\}$,
}
\item{$T_1(\mathsf{B}_{\bad}^\rightarrow) = T(\mathsf{B}_{\bad}) \cup T(\mathsf{B}_{\bad}^\rightarrow)$,
}
\item{$T_1(\mathsf{B}) = T(\mathsf{B})$, for all other boxes $\mathsf{B}$.}
\end{itemize}
The only two boxes which are different in $T_1$ and $T$ are $\mathsf{B}_{\bad}$ and $\mathsf{B}_{\bad}^\rightarrow$, so to confirm that $T_1$ remains a valid tableau we only need to check the following local region. 
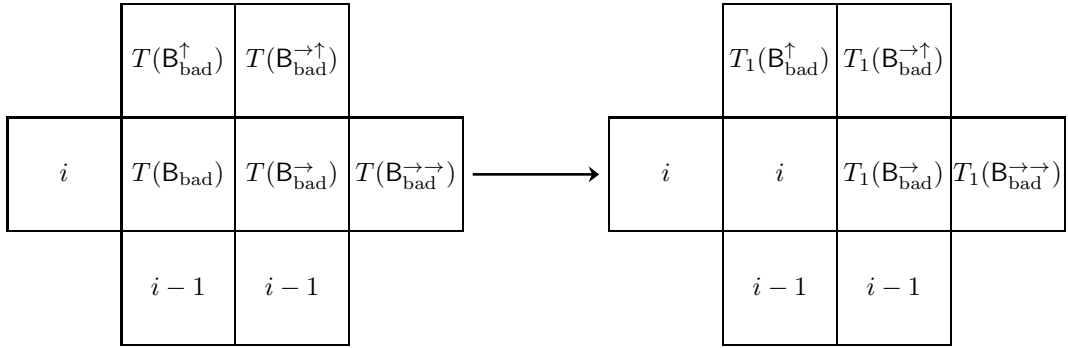
\[
\begin{tikzpicture}{>=latex}
\ytableausetup{boxsize=1.5cm}

\node (m) at (0,0) {\begin{ytableau}
\none & T(\mathsf{B}_{\bad}^\uparrow) & T(\mathsf{B}_{\bad}^{\rightarrow\uparrow}) & \none\\
i & T(\mathsf{B}_{\bad}) & T(\mathsf{B}_{\bad}^\rightarrow) & T(\mathsf{B}_{\bad}^{\rightarrow\rightarrow})\\
\none & i-1 & i-1
\end{ytableau}};
\node (f2m) at (8,0) {\begin{ytableau}
\none & T_1(\mathsf{B}_{\bad}^\uparrow) & T_1(\mathsf{B}_{\bad}^{\rightarrow\uparrow}) & \none\\
i & i & T_1(\mathsf{B}_{\bad}^\rightarrow) & T_1(\mathsf{B}_{\bad}^{\rightarrow\rightarrow})\\
\none & i-1 & i-1
\end{ytableau}};
\draw[-{Stealth[width=6pt,length=4pt]}, line width=1pt](m) -- (f2m);\end{tikzpicture}
\]
Most increasingness conditions follow immediately from inspection and $T \in \SVT(\mu,n)$; we only need to remark that since we only add elements (weakly) less than the minimum of $T(\mathsf{B}_{\bad}^\rightarrow)$ to $\mathsf{B}_{\bad}^\rightarrow$, we cannot introduce a violation with the boxes $\mathsf{B}_{\bad}^{\rightarrow \uparrow}$ or $\mathsf{B}_{\bad}^{\rightarrow\rightarrow}$. To conclude that $d(T) \leq d(T_1)$, observe that $|T(\mathsf{B}_{\bad}) \cap T(\mathsf{B}_{\bad}^\rightarrow)| \leq 1$, and so
\[|\{i\}| + |T(\mathsf{B}_{\bad}) \cup T(\mathsf{B}_{\bad}^\rightarrow)| \geq 1+(|T(\mathsf{B}_{\bad})|+|T(\mathsf{B}_{\bad}^\rightarrow)|-1).\]
Thus, the total content of these two boxes in $T_1$ is weakly larger than in $T$.

\bigskip
\noindent
{\sf (Case 2: $\mathsf{B}_{\bad}$ is the rightmost box in row $i$):}
In this case, we set $T_1(\BBB_\bad) = [i,n]$ and $T_1(\BBB) = T(\BBB)$ for all other boxes $\BBB$. Since $\BBB_\bad$ is in row $i$, a $\min(T(\BBB_\bad)) \geq i$, so $T(\BBB_\bad) \subseteq [i,n]$, and therefore $d(T) \leq d(T_1)$. The tableau $T_1$ is still a valid set-valued tableau, because by assumption $T_1(\mathsf{B}_{\bad}^\leftarrow) = \{i\}$ and $T_1(\mathsf{B}_{\bad}^\downarrow) = \{i-1\}$ (or is empty, if $i = 1$), and boxes in the other two directions do not exist. Since $T(\mathsf{B}_{\bad}) \subseteq [i,n]$, it clearly follows that $|T(\mathsf{B}_{\bad})| \leq |[i,n]|$, and since this is the only box whose content changes between $T$ and $T_1$, we conclude $d(T) \leq d(T_1)$.

\bigskip
\noindent
(There is no need for an analogous case to case 3 of the proof of Lemma~\ref{push}, since in the symmetric case there is no main diagonal and thus no varying behaviour if a box is the only box in its row or not.) Exactly as in Lemma~\ref{push}, repeating this construction yields a finite sequence of tableaux $T, T_1,\dots,T_j= \mathcal{N}_{\mu,n}$ such that $d(T) \leq d(T_1) \leq \dots \leq d(T_j)$, completing the proof.
\end{proof}

We may now prove the main results theorem of this section, complementing the formulas of \cite{RRR} and paralleling the formulas of Theorems~\ref{main} and~\ref{regularity_theorem}.

\begin{theorem}\label{thm:typeAdegree}
Let $\lambda$ be a a partition, and let $\mu = (\mu_1,\dots,\mu_\ell)$ be the largest strict partition contained in $\lambda$. Then
\begin{equation}
\deg(G_{\lambda,n}) = |\mu| + \ell n - \frac{\ell(\ell+1)}{2}.
\end{equation}
\end{theorem}
\begin{proof}
    By Lemma~\ref{lem:sym_degree_equality}, $\deg(G_{\lambda,n}) = \deg(G_{\mu,n})$. By Lemma~\ref{sympush}, $\deg(G_{\mu,n})$ equals the number of labels in the tableau \eqref{sym_maxtab}. Elementary counting of entries then yields the theorem.
\end{proof}

\begin{corollary}\label{cor:typeAreg}
       Let $w$ be a Grassmannian permutation with $w(n) > w(n+1)$ and $\shape(w) = \lambda$. Then the matrix Schubert variety $X_w$ has Castelnuovo--Mumford regularity \[\reg(S/I_w) = \ell n - \frac{\ell(\ell+1)}{2} - \big(|\lambda| - |\mu|\big).
       \]
\end{corollary}
\begin{proof}
    This follows from Theorem~\ref{thm:typeAdegree} exactly as in the analogous calculations of \cite[\S 4.2]{RRR}
\end{proof}

We end this section with some remarks connecting to the Grothendieck degree formula from \cite{PSW}. We thank Anna Weigandt for suggesting these ideas to us. To avoid introducing significant amounts of background and notation, this discussion is less self-contained than the rest of the paper; however, we include pointers to further elaboration in the literature.

A permutation $w$ is \emph{inverse fireworks} \cite[Definition~3.5]{PSW} if the initial elements of the maximal decreasing runs of $w^{-1}$ are in increasing order. For example, $u = 317429865$ has maximal decreasing runs $31$, $742$, and $9865$, whose initial elements $3,7,9$ appear in increasing order; hence $u^{-1}$ is inverse fireworks. For the definition of Grothendieck polynomials $\mathfrak{G}_w$ indexed by arbitrary permutations, see, e.g., \cite{Knutson.Miller,PSW}. A permutation $w$ is \emph{$k$-Grassmannian} if $w(i) < w(i+1)$ for all $i \neq k$. Given a partition $\lambda$ of length $\ell$ and $n \geq \ell$, let $w_{\lambda,n}$ be the unique $n$-Grassmannian permutation such that, for each $1 \leq i \leq n$, there exist exactly $\lambda_i$ values $v > n$ such that $w_{\lambda,n}(v) < w_{\lambda,n}(n + 1 - i)$. Every $n$-Grassmannian permutation is of this form for some partition $\lambda$.
The key fact is that $\mathfrak{G}_{w_{\lambda,n}} = G_{\lambda,n}$.

The proof for the degree formula in \cite{PSW} is essentially by reduction to the case where $w$ is inverse fireworks. The following connects this reduction to the corresponding reduction arguments of this section. For the definition of the \emph{Rothe diagram} $D(w)$ of a permutation $w$, see, e.g., \cite{PSW,Meszaros.Setiabrata.StDizier}. 

\begin{proposition}\label{prop:anna_sez}
    A Grassmannian permutation $w_{\lambda,n}$ is inverse fireworks if and only if $\lambda$ is a strict partition.
\end{proposition}
\begin{proof}
    Applying transpose to \cite[Proposition~3.9]{Meszaros.Setiabrata.StDizier} implies that a permutation $w$ is inverse fireworks if and only if the rightmost box of each nonempty row $i$ of $D(w)$ appears in position $(i, w(i) - 1)$. For an $n$-Grassmannian permutation $w_{\lambda,n}$ with $\lambda_i = \lambda_{i+1}$, it is straightforward from the definition to see that the rightmost boxes of rows $n+1-i$ and $n+1 - (i+1)$ of $D(w_{\lambda,n})$ appear in the same column. Hence, such a permutation cannot be inverse fireworks. Conversely, if $\lambda$ is a strict partition, it is similarly straightforward to see that each nonempty row of $D(w_{\lambda,n})$ satisfies the condition for $w_{\lambda,n}$ to be inverse fireworks.
\end{proof}

Proposition~\ref{prop:anna_sez} gives some explanation for the appearance of strict partitions in our analysis. Moreover, the reduction from an arbitrary partition $\lambda$ to the largest strict partition $\mu$ contained in $\lambda$ is likely related to the \emph{inverse fireworks map} $\Phi_{\mathrm{inv}}$ of \cite[\S 4.4]{PSW}; to avoid a major digression and because our arguments are easier than the more general arguments of \cite{PSW}, we do not pursue this line of inquiry further here. Proposition~\ref{prop:anna_sez}, together with the results of \cite{PSW}, suggests that there might be an appropriate notion of ``FPF-inverse fireworks fixed-point-free involutions'' governing the regularity of all skew-symmetric matrix Schubert varieties, for which D-partitions appear from the fixed-point-free involutions that are both FPF-vexillary and FPF-inverse fireworks. More generally, it suggests some hope of porting the entire theory of \cite{PSW} to the fixed-point-free involution setting.

\section{The case of \texorpdfstring{$Q$}{Q}-Grothendieck polynomials} \label{Q}
If we remove the final assumption in Definition~\ref{ptabdef} and allow primed entries on the main diagonal, then we obtain what are referred to as \emph{$Q$-shifted set-valued tableaux} \cite{INN,IN}. In particular, every $P$-shifted set-valued tableau is also $Q$-shifted. The generating function (analogous to Definition~\ref{pgrothendieckdef}) for $Q$-shifted set-valued tableaux is the \emph{$Q$-Grothendieck polynomial} $\INQ_{\lambda,n}$. Below, we use the notation $\QSVT(\lambda,n)$ to refer to the set of all $Q$-shifted set-valued tableaux of shape $\lambda$ with entries from $[n]_\mathbb{S}$. The $Q$-Grothendieck polynomials have similar geometric significance to other families of Grothendieck polynomials; they are representatives of $K$-theoretic Schubert classes in the \emph{Lagrangian Grassmannian} parametrizing isotropic $n$-planes in $\mathbb{C}^{2n}$ with respect to a nondegenerate skew-symmetric bilinear form. (For more background on the $K$-theory of Lagrangian Grassmannians, see \cite{Buch.Ravikumar,IN,Pechenik.Yong}.) Given the significance of the degrees of other families of Grothendieck polynomials to regularity questions, it is natural to ask if our proof can be adapted to obtain the $Q$-Grothendieck degrees as well. 

\begin{remark}
	Even if one succeeds in characterizing the degrees of $Q$-Grothendieck polynomials, there has not been developed an analogous body of theory to that applied in Section~\ref{regularity}, so it is unclear what if any regularities these degrees track. However, see some discussion at the end of \cite[$\mathsection 1$]{Marberg.Pawlowski:Grobner} for some potentially related ideas about the geometry of \emph{symmetric matrix Schubert varieties}. The geometry in this setting, however, appears to be much more difficult; see \cite{Pin,Wyser.Yong} for further discussion.
\end{remark}

The proof of Lemma~\ref{push} essentially does not depend on the fact that the tableau in question is $P$-shifted, rather than $Q$-shifted; the same proof, \textit{mutatis mutandis}, yields that a maximum degree $Q$-shifted tableau for a D-partition $\Delta$ is the same as $\mathcal{M}_{\Delta,n}$, except for containing the newly allowed $k'$ on the $k$th row of the main diagonal. Thus, we can still determine a tableau of maximum degree in the case that $\lambda$ is a D-partition. 

The proof of Lemma~\ref{squish}, however, \emph{does} rely on the tableau being $P$-shifted, as the analysis applied to ribbons of type (B) can fail when there are primes in the top box on the main diagonal. Indeed, the conclusion of Lemma~\ref{squish} does not hold in the $Q$-shifted case; there are examples where $\deg(\INQ_{\Delta,n}) < \deg(\INQ_{\lambda,n})$ with $\Delta$ the largest D-partition inside $\lambda$. One such example is the partition $\lambda = 421$. The largest D-partition contained in $421$ is $\Delta = 42$. However, the tableau $T \in \QSVT(42,3)$ shown below is of maximum degree $14$ in $\QSVT(42,3)$, but has lower degree than the tableau $T' \in \QSVT(421,3)$, which has degree $15$:
\[
\ytableausetup{boxsize=1.09cm}
\begin{ytableau}
\none & \none \\
\none[T =] & \none & 2'2 & 23'3\\
\none & 1'1 & 1 & 1 & 12'23'3
\end{ytableau}
\hspace{1cm}
\begin{ytableau}
\none & \none & \none & 3'3\\
\none[T'=] & \none & 2'2 & 23'\\
\none & 1'1 & 1 & 1 & 12'23'3
\end{ytableau}
\]
While it seems difficult to characterize when $\deg(\INQ_{\Delta,n}) < \deg(\INQ_{\lambda,n})$, we make a couple of observations.

\begin{proposition}\label{prop:sometimes_equality}
	If  $\Delta$ is the largest D-partition in $\lambda$ and $\ell(\Delta) = \ell(\lambda)$, then 
 \[
 \deg(\INQ_{\Delta,n}) = \deg(\INQ_{\lambda,n}).
 \]
\end{proposition}
\begin{proof}
    In this case, the proof of Lemma~\ref{push} extends directly, since
a ribbons of type (B) (as defined in that proof) only occur when an entire row of the tableau is deleted. 
\end{proof}

 Note that Proposition~\ref{prop:sometimes_equality} gives only a sufficient condition for equality and is not a characterization. For example, one can compute that $\deg(\INQ_{321,3}) = \deg(\INQ_{31,3}) = 12$, while $31$ is the largest D-partition inside $321$.

It is tempting to attempt to describe the degrees of $Q$-Grothendieck polynomials in terms of the degrees of $P$-Grothendieck polynomials. We observe some bounds on the difference between these degrees.
\begin{proposition}\label{prop:QvsP}
	For any strict partition $\lambda$, we have \[
 \deg(\INQ_{\lambda,n}) - \deg(\IN_{\lambda,n}) \in [\ell(\lambda),n].
 \]
 In particular, if $n = \ell(\lambda)$, then $\deg(\INQ_{\lambda,n}) = \deg(\IN_{\lambda,n}) + n$.
\end{proposition}
\begin{proof}
	 To establish the lower bound, note that for any $P$-shifted set-valued tableau $T$ we can produce a $Q$-shifted tableau $T'$ with $d(T') = d(T)+\ell(\lambda)$ by adding to every box on the main diagonal a primed copy of the minimum element it contains. 
	 
	 For the upper bound, consider a tableau $S \in \QSVT(\lambda,n)$. One can see that if $S$ is of maximal degree, then every box on the main diagonal must contain both a primed and an unprimed entry. In particular, we can produce a $P$-shifted set-valued tableau $S^\dagger$ by deleting all primed entries from the main diagonal of $S$. It therefore suffices to observe that no tableau in $\QSVT(\lambda,n)$ can have any particular primed value $i'$ in more than one box on the main diagonal. This is immediate, because if a box $\mathsf{B}$ is northeast of a box $\mathsf{A}$ in a shifted set-valued tableau, then $\min(T\mathsf{(B)}) > \max(T\mathsf{(A)})$. 
	 Thus, $\deg(\INQ_{\lambda,n}) \leq \deg(\IN_{\lambda,n}) + n$.
\end{proof}


Recently, Y.~Chiu and E.~Marberg \cite{Chiu.Marberg} proved a result giving a signed and cancellative expansion of a $Q$-Grothendieck polynomial in terms of $P$-Grothendieck polynomials. As the formula is signed, it cannot be used directly to compute the degree of a $Q$-Grothendieck polynomial in $n$ variables by specialization; in fact, the examples discussed above show that some highly coordinated cancellations can sometimes occur when specializing this formula to particular numbers of variables, causing entire leading terms of the sum to vanish. It appears difficult to understand when these cancellations occur.

\section*{Acknowledgements}
Both authors were partially supported by NSERC Discovery Grant RGPIN-2021-02391 and Launch Supplement DGECR-2021-00010. M.S. was also partially supported by an NSERC Canada Graduate Scholarship--Master's (CGS-M) during the writing of this paper. O.P.\ thanks Zach Hamaker, Patricia Klein, and Jenna Rajchgot for helpful conversations. We are also happy to thank Patricia Klein, Jenna Rajchgot, and Anna Weigandt for very useful comments on a draft of this document.

\bibliographystyle{amsalphavar}
\bibliography{pshifted}

\end{document}